\documentclass[12pt]{amsart}
\usepackage[utf8]{inputenc}
\usepackage[all]{xy}
\usepackage[english]{babel}
\usepackage{amssymb, enumerate}
\usepackage{mathrsfs}
\usepackage[lite, initials]{amsrefs}
\usepackage{geometry}
\usepackage{hyperref}

\newtheorem{teorema}{Theorem}[section]
\newtheorem*{theorem*}{Main Theorem}
\newtheorem{lemma}[teorema]{Lemma}
\newtheorem{propos}[teorema]{Proposition}
\newtheorem{corol}[teorema]{Corollary}
\theoremstyle{definition}
\newtheorem{ex}{Example}[section]
\newtheorem{rem}{Remark}[section]
\newtheorem{defin}[teorema]{Definition}

\calclayout

\newcounter{tempo}

\def\R{{\mathbb R}}

\def\C{{\mathbb C}}

\def\P{{\mathbb P}}
\def\J{{\mathbb J}}
\def\CP{{\mathbb{CP}}}
\def\H{{\mathbb H}}
\def\cU{{\mathcal U}}

\def\sfera{{\mathbb S}}

\newcommand{\tr}{\mathrm{tr}}

\newcommand{\End}{\mathrm{End}}
\newcommand{\Uno}{\mathbb{I}}

\def\Span{{\rm Span}}

\def\de{\partial}

\title[The twistor space of a real associative algebra]{The twistor space of a real associative algebra: incidence geometry, semisimple classification and slice-regular functions}
\author[S. Mongodi]{Samuele Mongodi}
\address{Politecnico di Milano, Dipartimento di Matematica, Via Bonardi, 9 -- I-20133 Milano, Italy}
\email{samuele.mongodi@polimi.it}
\date\today

\begin{document}
\begin{abstract}
To every finite-dimensional real associative algebra $A$ we associate its \emph{twistor space} $S=S(A)$, the real algebraic set of square roots of $-1$ in $A$. Left multiplication endows $S$ with an almost complex structure $J_S$, and associativity of $A$ is exactly what makes $J_S$ integrable: the resulting complex manifold embeds biholomorphically into the complex Grassmannian of $\C\otimes A$, via the assignment of $s\in S$ to the $(-i)$-eigenspace of the complexification of the left multiplication $L_s$.

This identification turns the tautological pairing $\pi:\C\otimes A\times S\to A$ into an incidence correspondence: for every compact complex subvariety $K\subseteq S$, the associated zero variety $Z_K\subset\C\otimes A$ is a complex analytic set, by Remmert's proper mapping theorem, and the pull-back of the incidence variety along a holomorphic lift governs the zero set of the corresponding function on $\pi(\cU\times K)$. In this setting we also give an intrinsic, section-free reformulation of the twistor transform of Gentili, Salamon and Stoppato as a holomorphic map into the projective bundle $\mathbb P(\mathcal V\oplus\mathcal V)$ over $S$.

After fixing a Euclidean structure on $A$, we single out the compact subvariety $S_0\subset S$ of square roots of $-1$ whose left multiplication is orthogonal, and we classify it for semisimple algebras through the Wedderburn decomposition: $S_0$ is a finite union of products of compact Hermitian symmetric spaces ($O(2m)/U(m)$, $Sp(m)/U(m)$ and Grassmannians $U(n)/(U(j)\times U(n-j))$); in the classical simple cases we compute explicit equations for the associated Euclidean zero variety $Z_0$, realizing it as an isotropic or determinantal cone.

This geometric picture originates in, and is illustrated throughout by, the theory of slice-regular functions on real associative algebras: for $A=\H$, the twistor space $S(\H)$ is the classical twistor sphere and the incidence correspondence above recovers the twistor transform of Gentili, Salamon and Stoppato. As an application of the general machinery, we finally introduce a broader class of generalized slice-regular functions, for a compact and connected parameter space $S$, and show that the maximum modulus principle, a representation formula and a twisted Cauchy--Riemann characterization extend to this class via a reparametrization of $S$ by its holomorphic automorphisms.
\end{abstract}

\maketitle

\tableofcontents

\section{Introduction}

Given a real vector space $V$, the set $\mathcal E_V$ of linear complex structures on $V$ is itself a complex manifold, biholomorphic to an open subset of a complex Grassmannian; this elementary fact is the starting point of twistor theory, from the twistor algebra of Penrose \cite{Pen} and the twistor spaces of self-dual Riemannian $4$-manifolds of Atiyah, Hitchin and Singer \cite{AHS}, to the twistor spaces of Hermitian symmetric spaces studied by Salamon \cite{Sal}. In this paper we study the analogous, but purely algebraic, construction attached to a finite-dimensional real associative algebra $A$: the set
$$S=S(A)=\{a\in A\ :\ a^2=-1\}$$
of square roots of $-1$ in $A$, which we call the \emph{twistor space} of $A$. This generalizes, from the complex numbers, the choice of an imaginary unit: in $\C$ there are exactly two square roots of $-1$, essentially indistinguishable as complex structures, since $\pm i$ share the same eigenspaces in $\C\otimes\C$; in a generic real associative algebra $A$, the set $S(A)$ is typically much larger and the induced family of complex structures correspondingly richer.

\medskip

Every $s\in S$ defines a complex structure $L_s$ on $A$ by left multiplication, and $S$ inherits from the space $\mathcal E_A$ of linear complex structures on $A$ a natural almost complex structure $J_S$; associativity of $A$ is precisely what makes $J_S$ integrable (see Proposition \ref{prp_integrabile}). Since the inclusion of $S$ into $\mathcal E_A$ is then holomorphic, $S$ is identified with a complex subvariety of the Grassmannian of $\C\otimes A$ (see Theorem \ref{teo_cplstr}).

\begin{theorem*}[A --- Complex geometry of $S(A)$]
Let $A$ be a finite-dimensional real associative algebra. The twistor space $S=S(A)$ is a smooth real algebraic subset of $A$ (Lemma \ref{lemma_S_smooth}), possibly with components of different dimensions, and the almost complex structure $J_S$ induced by left multiplication is integrable (Proposition \ref{prp_integrabile}). The map $s\mapsto\mathcal W(L_s)$, sending $s$ to the $(-i)$-eigenspace of the complexification of $L_s$, is a biholomorphism of $(S,J_S)$ onto a complex submanifold of the Grassmannian $\mathrm{Gr}_\C(\dim_\R A/2,\C\otimes A)$ (Theorem \ref{teo_cplstr}).
\end{theorem*}

\medskip

This complex-geometric identification turns the tautological pairing $\pi:\C\otimes A\times S\to A$, $\pi(w,s)=a+sb$ for $w=1\otimes a+i\otimes b$, into an incidence correspondence in $\C\otimes A\times S$: the fibers $\mathfrak Z_K(a)=\{(w,s)\in\C\otimes A\times K: \pi(w,s)=a\}$, for a fixed $a\in A$ and a compact complex subvariety $K\subseteq S$, form a real-analytic foliation of $\C\otimes A\times K$ with complex leaves, whose projections $Z_K(a)\subset\C\otimes A$ are complex analytic sets by Remmert's proper mapping theorem (see Theorem \ref{teo_zero_var} and Proposition \ref{prp_zero}). Once $u\in S$ is fixed, $A$ acquires the structure of a complex vector space via $L_u$; extending the conjugation action of the invertible elements of $A$ on $S$ to the zerodivisors produces a holomorphic vector bundle $\mathcal V$ over $S$ (Proposition \ref{prp_bundle}), and a never-vanishing section of $\mathcal V$ over an affine chart of $S$ yields a generalized stereographic parametrization and an associated twistor transform (Proposition \ref{prp_twist}); dropping the need for a global section, the same data define an intrinsic holomorphic map into the projective bundle $\P(\mathcal V\oplus\mathcal V)$ (Theorem \ref{teo_intrinsic_twistor}).

\begin{theorem*}[B --- Twistor/incidence geometry for compact $K\subseteq S$]
Let $A$ be a finite-dimensional real associative algebra and let $K\subseteq S(A)$ be a compact complex subvariety. Then:
\begin{enumerate}
\item[(i)] the zero variety $Z_K=\{w\in\C\otimes A: \pi(w,s)=0\ \text{for some}\ s\in K\}$ is a complex analytic subset of $\C\otimes A$ (Theorem \ref{teo_zero_var});
\item[(ii)] for every holomorphic lift $F:\cU\to\C\otimes A$ of a slice-regular function $f$ on $A$, the incidence correspondence between $F^{-1}(Z_K)$ and $K$ is analytic, and it governs the structure of the fibers of $f$ over $\pi(\cU\times K)$ (Theorem \ref{teo_incidence});
\item[(iii)] there is a holomorphic map $S\to\P(\mathcal V\oplus\mathcal V)$, defined intrinsically without any choice of global section of $\mathcal V$, which for $A=\H$ recovers the twistor transform of Gentili, Salamon and Stoppato \cite{GSS} (Theorem \ref{teo_intrinsic_twistor}).
\end{enumerate}
\end{theorem*}

\medskip

The three parts of Main Theorem B are conveniently pictured as a single double fibration
\[
\xymatrix{
& \mathfrak Z_K \ar[dl]_-{p_K} \ar[dr]^-{q_K} & \\
\C\otimes A & & K
}
\]
in direct analogy with the double fibration of Penrose's twistor correspondence between (compactified) spacetime and twistor space \cite{Pen}: the correspondence space $\mathfrak Z_K$ projects, via the proper holomorphic map $p_K$, onto the zero variety $Z_K=p_K(\mathfrak Z_K)\subset\C\otimes A$ (part (i)), and, via $q_K$, onto $K\subseteq S$. Pulling this diagram back along a holomorphic lift $F:\cU\to\C\otimes A$ replaces $\C\otimes A$ by $\cU$ and $\mathfrak Z_K$ by the incidence correspondence $\mathcal C_{F,K}$ of part (ii); dropping the choice of a section along the fibres of $q_K$, instead, produces the intrinsic twistor map of part (iii). We keep this picture, rather than the classical language of stem functions, at the centre of the paper.

\medskip

Fixing a Euclidean inner product on $A$, we single out the compact subset $S_0\subset S$ of those square roots of $-1$ whose left multiplication is orthogonal; $S_0$ is a compact complex subvariety of $S$ (Theorem \ref{teo_S0_complex}) and carries an associated Euclidean quadratic cone $Q_0$, which should not be confused, in general, with the quadratic cone of Ghiloni and Perotti: the two agree only under additional compatibility assumptions with a chosen $*$-involution. Compactness of $S_0$ makes it possible to describe the associated zero variety $Z_0\subset\C\otimes A$ by global holomorphic equations (Theorem \ref{teo_luogo}), and, when $A$ is semisimple, to classify $S_0$ completely.

\begin{theorem*}[C --- Semisimple classification and determinantal equations]
Let $A$ be a finite-dimensional semisimple real associative algebra, with Wedderburn decomposition $A\cong\prod_r M_{n_r}(D_r)$, $D_r\in\{\R,\C,\H\}$. Then $S_0(A)$ decomposes as a finite union of products of compact Hermitian symmetric spaces: each factor $S_0(M_{n_r}(D_r))$ is $O(2m)/U(m)$ if $D_r=\R$ and $n_r=2m$ is even, a disjoint union of Grassmannians $U(n_r)/(U(j)\times U(n_r-j))$ if $D_r=\C$, and $Sp(n_r)/U(n_r)$ if $D_r=\H$ (Theorem \ref{teo_wedderburn}). In each of these simple cases the associated zero variety $Z_0$ coincides with the set-theoretic locus $Z$ defined by the vanishing of an explicit quadratic (Clifford) equation, and is realized as an isotropic or determinantal cone over the corresponding symmetric space (Theorem \ref{teo_luogo} and \S\ref{explicit_Z0}).
\end{theorem*}

\medskip

Beyond these three results, we develop several further points. We make the incidence picture of Main Theorem B modular in $K$, relating it to the study of zeros and singularities of slice functions of Ghiloni, Perotti and Stoppato \cites{GPS1,GPS2}, to the several-variable theory \cite{GPsv} and to the local representation formula of Gentili and Stoppato \cite{GeSt}. Our explicit computation of $Z_0$ in the classical simple cases $\H$, $M_{2m}(\R)$, $M_n(\C)$ and $M_m(\H)$ (\S\ref{explicit_Z0}) rests, in part, on the behaviour of slice functions on domains without real points studied by Altavilla \cite{Alt}.

\bigskip

This geometric picture originates in, and remains best illustrated by, the theory of slice-regular functions. Slice-regular quaternionic functions were introduced by Gentili and Struppa \cite{GS2}; since then their theory has been extensively studied (see \cites{GSS1, CSS1, CSS2}) and extended to Clifford algebras \cite{CSS} and to general real alternative algebras \cite{GP1}. These functions share an impressive number of properties with holomorphic functions: power series expansions, a maximum modulus principle, a Cauchy formula, a rigid structure of the zero set, and so on. Ghiloni and Perotti \cite{GP1}, building on an idea going back to Fueter \cite{Fue}, for the quaternionic case, and to Sce \cite{Sce} and Rinehart \cite{Rin} for a general algebra, showed that every slice-regular quaternionic function is induced by a holomorphic function $F$ from an open subset of $\C$ to $\C\otimes A$; Gentili, Salamon and Stoppato \cite{GSS}, working with the classical twistor sphere $S(\H)=\CP^1$, associated to every quaternionic slice-regular function a holomorphic map from $\cU\times\CP^1$ to $\CP^3$, and the author \cite{M1} showed that the values of $f$ and of $F$ are linked by a family of complex hypersurfaces $Z(q)\subset\C^4$. With this complex structure on $S=S(A)$, a slice-regular function $f$ on $A$ (defined on the quadratic cone $Q_A$ of Ghiloni and Perotti \cite{GP1}) induces precisely a holomorphic map $\mathscr F$ from an open subset of $\C\times S$ to $\C\otimes A\times S$, of the special form $\mathscr F(z,s)=(F(z),s)$ (see Proposition \ref{prp_tensore}): it is the abstract counterpart, for a general algebra $A$, of the classical quaternionic correspondence recalled in full, together with the explicit link between the Ghiloni--Perotti stem function and the twistor transform of \cite{GSS}, in Appendix \ref{hol_quat}.

\medskip

From this description, many qualitative properties of slice-regular functions can be traced back to the holomorphicity of $\mathscr F$; however, the theory developed so far assumes the stronger hypothesis that $\mathscr F$ is exactly of the form $\mathscr F(z,s)=(F(z),s)$, with $F$ holomorphic. If one drops this restriction and considers instead an arbitrary holomorphic map $\mathfrak F:\cU\times S\to\C\otimes A\times S$ with the appropriate symmetries, a larger class of \emph{generalized slice-regular functions} is defined, to which much of the holomorphic machinery developed in this paper still applies; we discuss this application in Section \ref{application}.

\bigskip

The content is organized as follows. Section \ref{twistor_space} introduces the twistor space $S=S(A)$, its natural complex structure $J_S$, and the elementary case of slice-regular functions on $A$ (Main Theorem A). Section \ref{zero_var} describes the incidence variety and the zero variety $Z_K$ associated with a compact complex subvariety $K\subseteq S$, together with the relative incidence correspondence of a slice-regular function over such a $K$ (Main Theorem B, part (i)-(ii)). Section \ref{stereo_twistor} constructs the generalized stereographic parametrization and twistor transform of $S$, together with its intrinsic, section-free formulation (Main Theorem B, part (iii)). Section \ref{orto} restricts attention to the compact subvariety $S_0\subseteq S$ of orthogonal square roots of $-1$, proves the classification of $S_0$ for semisimple algebras and computes explicit equations for $Z_0$ in the classical simple cases (Main Theorem C). Section \ref{application} presents, as an application of this machinery, the class of generalized slice-regular functions and shows that the maximum modulus principle, a representation formula, an identity principle and a twisted Cauchy--Riemann characterization extend to it. Finally, Appendix \ref{hol_quat} collects, for the reader's convenience, the classical quaternionic picture -- the various definitions of slice-regular function, the results of \cite{GP1}, and the explicit correspondence with the twistor transform of \cite{GSS}, together with its link to \cite{M1} -- from which the constructions of this paper originated.

\bigskip

We expect the twistor space $S(A)$ to admit applications well beyond the theory of slice-regular functions, in at least four directions. First, since $S(A)$ is a purely linear-algebraic construction attached to a single algebra $A$, it can be globalized fibrewise: given a bundle of associative algebras $\mathcal A\to M$ over a manifold $M$ — for instance the Clifford bundle of a Riemannian manifold, or the endomorphism bundle of a vector bundle carrying extra structure — the associated bundle $S(\mathcal A)\to M$ is a natural candidate twistor space for the underlying geometric structure on $M$. For $A=\mathbb H$ this recovers the twistor spaces of self-dual $4$-manifolds of Atiyah, Hitchin and Singer \cite{AHS} and of quaternionic Kähler manifolds of Salamon \cite{Sal}; for other Clifford algebras it suggests generalized "Clifford-Kähler" geometries, possibly related to $\mathrm{Spin}(7)$- and $G_2$-structures via the spin representations of $\mathrm{Cl}(7)$ and $\mathrm{Cl}(8)$. Second, in the semisimple case $S_0(A)$ is, by Main Theorem C, a product of compact Hermitian symmetric spaces — exactly the class of homogeneous spaces for which the Penrose transform of representation theory, in the sense of Baston and Eastwood, computes the cohomology of homogeneous vector bundles via Borel–Weil–Bott; the double fibration $\C\otimes A\leftarrow\mathfrak Z_K\to K$ of this paper is precisely the kind of correspondence such a transform requires, and it would be interesting to make this dictionary precise. Third, the family of Cauchy–Riemann–Fueter-type operators $\mathscr D_\phi$ appearing in Remark \ref{rem_global_operator} suggests a Clifford-analytic theory of "monogenic" functions on a general algebra $A$, parallel to slice-regularity but governed by $S$ (or $S_0$) rather than by a single imaginary unit, with integral representation formulas of Penrose-transform type. Finally, the explicit determinantal and isotropic-cone equations for $Z_0$ obtained in \S\ref{explicit_Z0} have the same algebraic shape as the ADHM equations for instanton moduli, hinting at a possible bridge between the incidence geometry of $S_0(A)$ and moduli spaces in gauge theory. We plan to return to these questions elsewhere.

\section{The twistor space of a real associative algebra}\label{twistor_space}

The aim of this section is to extend the construction of the complex structure $\J$ in \eqref{eq_cpstr} to a more general setting.

Let $A$ be an associative real algebra with unity; as a real vector space, $A$ is isomorphic to $\R^N$. We consider the set
$$S=\{a \in A\ :\ a^2=-1\}\;.$$
We define some families of linear operators on $A$.
\begin{defin}Given $a\in A$, we define $L_a,\ R_a\in\mathrm{End}(A)$ as
$$L_a(x)=ax\qquad R_a(x)=xa\;.$$
We denote their sum by $F_a$, i.e. $F_a(x)=ax+xa$, and their composition by $K_a$, i.e. $K_a(x)=axa$.
\end{defin}

The following properties are easy to prove.

\begin{lemma}\label{lmm_conti}For $a\in A$, we have
\begin{enumerate}
\item $[R_a,L_a]=0$,
\item $F_a^2=L_a^2+R_a^2+2K_a$.\setcounter{tempo}{\value{enumi}}
\end{enumerate}
For $s\in S$, we have
\begin{enumerate}\setcounter{enumi}{\value{tempo}}
\item $R_s^2=L_s^2=-I$, therefore $R_s$ and $L_s$ are invertible,
\item $R_sF_s=K_s-I$,
\item $F_s^2=2K_s-2I$,
\item $K_s^2=I$.
\end{enumerate}
\end{lemma}

We start by exploring the geometry of the set $S$.

\begin{lemma}\label{lemma_S_smooth}
The set $S$ is a real algebraic subset of $A$ whose underlying real analytic space is a smooth embedded submanifold of $A$, possibly with connected components of different dimensions. Moreover,
\[
T_sS=\{h\in A\ :\ sh+hs=0\}
\]
for every $s\in S$, where $T_sA$ is identified with $A$ in the usual way.
\end{lemma}
\begin{proof}
Let $f:A\to A$ be the polynomial map $f(a)=a^2+1$. Then $S=f^{-1}(0)$, so $S$ is a real algebraic subset of $A$, and
\[
Df_s(h)=sh+hs
\]
for $s\in S$. Hence the expected tangent space is $\ker F_s$, where $F_s=L_s+R_s$.

We first recall that the group $A^\times$ of invertible elements of $A$ is a Zariski open dense subset of $A$. Indeed, the map $x\mapsto L_x$ is real linear and $x$ is invertible if and only if $L_x$ is invertible. Thus
\[
A^\times=\{x\in A:\det L_x\neq0\},
\]
and the polynomial $\det L_x$ is not identically zero, because $\det L_1=1$.

Fix $s\in S$. We show that $S$ is locally the conjugacy orbit of $s$. If $t\in S$ is sufficiently close to $s$, then
\[
y=1-ts
\]
is invertible. Moreover
\[
ys=(1-ts)s=s+t=t(1-ts)=ty,
\]
so $t=ysy^{-1}$. Thus all points of $S$ near $s$ are obtained by conjugating $s$.

Let
\[
E_+(s)=\{h\in A:\ sh+hs=0\},\qquad
E_-(s)=\{h\in A:\ sh-hs=0\}.
\]
The involution $K_s(h)=shs$ satisfies $K_s^2=I$, and $E_+(s)$ and $E_-(s)$ are respectively the $+1$ and $-1$ eigenspaces of $K_s$. Hence
\[
A=E_+(s)\oplus E_-(s).
\]
Consider the map, defined for $h\in E_+(s)$ sufficiently small,
\[
\psi_s(h)=(1+h)s(1+h)^{-1}.
\]
It takes values in $S$ and
\[
D\psi_s|_0(h)=hs-sh=-2sh.
\]
Since $h\mapsto -2sh$ is an automorphism of $E_+(s)$, the inverse function theorem shows that $\psi_s$ parametrizes an embedded submanifold of $A$ with tangent space $E_+(s)$. By the previous paragraph, this submanifold contains a neighbourhood of $s$ in $S$ and is contained in $S$. Therefore $S$ is smooth near $s$ and
\[
T_sS=E_+(s)=\{h\in A:\ sh+hs=0\}.
\]
\end{proof}

\begin{rem}\label{rem_componenti}
Since $K_s$ is an involution of $A$ with $K_s^2=I$, we have the eigenspace decomposition $A=E_+(s)\oplus E_-(s)$ of Lemma \ref{lemma_S_smooth}, and $\tr K_s=\dim E_+(s)-\dim E_-(s)$ is an integer with the same parity as $N$ and $|\tr K_s|\le N$. As $T_sS=E_+(s)$ by Lemma \ref{lemma_S_smooth}, the local real dimension of $S$ at $s$ is $\dim E_+(s)=(N+\tr K_s)/2$. The largest possible value of this dimension is attained when $\tr K_s$ is as large as possible; the case $\tr K_s=0$, in which $\dim E_+(s)=\dim E_-(s)=N/2$, will turn out to be the generic one in the examples below (see Remark \ref{rem_generic} and Corollary \ref{cor_Cliff2}).
\end{rem}

\begin{lemma}\label{lmm_tr}For $s\in S$, the operators $L_s$, $R_s$, $F_s$ have zero trace.\end{lemma}
\begin{proof}By Lemma \ref{lmm_conti}, we have that $L_s^2=R_s^2=-I$, therefore $L_s$ and $R_s$ have eigenvalues $\pm i$; as the trace has to be real, the only possibility is $\tr L_s=\tr R_s=0$.

Similarly, $F_s^2=2K_s-2I$ and $F_s^2+4I=2K_s+2I$, therefore
$$F_s^2(F_s^2+4I)=4(K_s+I)(K_s-I)=0$$
so $F_s$'s eigenvalues can be $0$, $\pm 2i$; again, as the trace is a real number, the only possibility is $\tr F_s=0$.\end{proof}


\begin{rem}\label{rem_iperpiano}From Lemma \ref{lmm_tr}, $S$ is contained in the kernel of the linear map $A\ni a\mapsto \tr(F_a)\in\R$, whereas $\tr(F_1)=2N$, so $S$ is contained in a hyperplane that does not contain $1$, hence the map $(x+iy, s)\mapsto x+sy$ is injective when $y>0$.\end{rem}

In the case of a Clifford algebra, the previous remark follows also from a more precise computation of the trace of the operator $L_s$.

\begin{corol}\label{cor_Cliff1}If $A$ is a Clifford algebra, $\tr L_s=0$ for all $s\in S$, i.e. $S$ is contained in a hyperplane which does not contain $1$. In particular, this implies that the map $(x+iy,s)\mapsto x\cdot1+y\cdot s$ is injective from $\C_+\times S$ into $A$.\end{corol}
\begin{proof} Let $e_1,\ldots, e_n$ be generators of $A$; a basis of $A$ as a real vector space is
$$\{e_I\ \vert\ I\subseteq\{1,\ldots, n\}\}$$
where, for $I=(i_1,\ldots, i_p)$,
$$e_I=e_{i_1}\cdots e_{i_p}\qquad e_0=e_{\emptyset}=1\;,$$
so, the dimension $N$ of $A$ is $2^n$.

For $a\in A$, we have that $a=\sum a_Ie_I$, then, the $e_J$-component of the element
$$L_a(e_J)=\sum a_Ie_Ie_J$$
is $a_0$, because the only $I$ such that $e_Ie_J=e_J$ is $I=\emptyset$, as all the elements of the basis are invertible. Therefore
$$\tr L_a=2^na_0\;.$$
In the same way, we prove that $\tr R_a=2^na_0$. As, for $s\in S$,  $0=\tr F_s=\tr L_s+\tr R_s=2^{n+1}s_0$, then $S$ is contained in the hyperplane $s_0=0$.

Given that $1$ is contained in the hyperplane $s_0=1$, we have that two elements of $A$ of the form $x+sy$ and $u+sv$ with $x,u\in\R$ and $y,v\in\R_+$ coincide if and only $x=u$ and $y=v$.
\end{proof}

\begin{rem}\label{rem_generic}
The case $\tr K_s=0$ is, in some sense, generic. If $h\in T_sS$, then the map $L_h:A\to A$, given by $L_h(x)=hx$, is such that $L_h(E_+(s))\subseteq E_-(s)$ and $L_h(E_-(s))\subseteq E_+(s)$; if $h$ is invertible, this implies that $\dim E_+(s)=\dim E_-(s)$, hence $\tr K_s=0$. Therefore, if $T_sS$ contains at least one invertible element, then $\tr K_s=0$ and $s$ lies on a connected component of $S$ of dimension $N/2$. We stress that this condition is unrelated, in general, to the metric subset $S_0$ introduced in Section \ref{orto}: the former is intrinsic to the algebra, while the latter depends on a choice of Euclidean structure.
\end{rem}

For every $s\in S$, the map $L_s:A\to A$ can be identified with an element of $\End(T_sA)$; moreover, if $h\in T_sS$, then $sh+hs=0$ and
$$sL_s(h)+L_s(h)s=ssh+shs=s(sh+hs)=0$$
so $L_s$ restricts to an endomorphism of $T_sS$. We define the almost complex structure $J_S:TS\to TS$ as
$$J_S(s,h)=(s,L_s(h))\;.$$

\begin{propos}\label{prp_integrabile}The structure $J_S$ is integrable.\end{propos}
\begin{proof}Let $X$, $Y$ be vector-fields on $S$. The Nijenhuis tensor is then
$$N(X,Y)=[X,Y]+J_S([J_SX,Y]+[X,J_SY])-[J_SX,J_SY]\;.$$
As it is well known, $[X,Y]=dX(Y)-dY(X)$, where the differential is computed as the differential of the maps $X,Y:S\to A$; so we need to compute the quantities
$$d(J_SX)(J_SY)=J_SY\cdot X + J_S(dX(J_SY))$$
$$J_S(d(J_SX)(Y))=J_S(Y\cdot X) -dX(Y)\;,$$
where $\cdot$ denotes the product in $A$, and substitute them in the formula for the Nijenhuis tensor, obtaining
$$N(X,Y)=J_S(Y\cdot X)-J_S(X\cdot Y)-J_SY\cdot X+J_SX\cdot Y\;.$$
As $N(\cdot, \cdot)$ is a tensor, the value of $N(X,Y)$ at a point $s$ is determined by the values of $X(s)$ and $Y(s)$; so, if $X(s)=x\in A$ and $Y(s)=y\in A$, then
$$N(X,Y)(s)=s(yx)-s(xy)-(sy)x+(sx)y=0$$
because $A$ is associative.
Therefore the almost complex structure $J_S$ is integrable.\end{proof}

The following corollary is now trivial.

\begin{corol}Every connected component of $S$ is a complex manifold.\end{corol}

\begin{rem}Let $S_j=\{s\in S\ :\ \tr K_s=j\}$, then $S_j=\emptyset$ if $j$ is odd or $|j|>N$. We note that each $S_j$ is a complex manifold (maybe disconnected), which is invariant under all internal automorphisms of $A$; moreover $\dim_\R S_j=(N+j)/2$ if it is non-empty.
\end{rem}

\begin{rem}In case $A=\H$, then $N=4$ and $S$ is the unit sphere of $\R^3$; the complex structure $J_S$ is the standard one, induced on $\sfera^2\subset\R^3$ by the vector product. We could also carry on the same computations in the case of the algebra of octonions, but then the almost complex structure is not integrable.\end{rem}

\begin{corol}\label{cor_Cliff2}If $A$ is a Clifford algebra of signature $(p,q)$ with $p-q\not\equiv 3\bmod 8$, all the connected components of $S$ are of real dimension $N/2$.\end{corol}
\begin{proof}It is enough to show that $\tr K_s=0$ for all $s\in S$. This follows because
\begin{equation}\label{eq_traccia}\tr K_s=2^ns_0^2\;,\end{equation}
so, by Corollary \ref{cor_Cliff1}, $\tr K_s=0$.

We now prove Equation \eqref{eq_traccia}. We consider the basis defined in Corollary \ref{cor_Cliff1}, we fix $s=\sum s_Ie_I$ and we compute
$$K_s(e_J)=\sum_{I,K} s_Is_Ke_Ie_Je_K;.$$
The $e_J$-component of such a sum is clearly given by
$$\sum_I s_I^2e_Ie_Je_I\;,$$
therefore
$$\tr K_s=\sum_J\sum_I s_I^2e_I^2k(I,J)$$
where $k(I,J)\in\{1,-1\}$ is such that
$$e_Je_I=k(I,J)e_Ie_J\;.$$
We note that
$$k(I,J)=(-1)^{|I||J|}(-1)^{|I\cap J|}$$
so, we define $i=|I|$, $k=|I\cap J|$, $j+k=|J|$. We rewrite the previous sum as
$$\tr K_s=\sum_{I}s_I^2e_I^2\sum_{k=0}^i\sum_{j=0}^{n-i}\binom{i}{k}\binom{n-i}{j}(-1)^{i(j+k)}(-1)^k=$$
$$=\sum_{I}s_I^2e_I^2\sum_{k=0}^i\binom{i}{k}(-1)^{(i+1)k}\sum_{j=0}^{n-i}\binom{n-i}{j}(-1)^{ij}=\sum_{I}s_I^2e_I^2((-1)^{i+1}+1)^i((-1)^i+1)^{n-i}$$
$$=\left\{\begin{array}{lcr}2^ns_0^2&\textrm{ if }&2\mid n\\2^ns_0^2+2^n\omega^2s_\omega^2&\textrm{ if }&2\not\mid n\end{array}\right.\;,$$
where $\omega=e_1\cdot\ldots\cdot e_n$ and $s_\omega$ is the corresponding coefficient. Suppose that the Clifford algebra has $p$ generators that square to $1$ and $q$ that square to $-1$, so that $n=p+q$. It is an easy computation to show that $\omega^2=-1$ if $p-q\equiv 3\bmod 4$ and $\omega^2=+1$ if $p-q\equiv 1\bmod 4$. So
$$\tr K_s=\left\{\begin{array}{lcr}2^ns_0^2&\textrm{ if }&2\mid p-q\\2^n(s_0^2-s_\omega^2)&\textrm{ if }&p-q\equiv 3\bmod 4\\2^n(s_0^2+s_\omega^2)&\textrm{ if }&p-q\equiv 1\bmod 4\end{array}\right.\;.$$
Moreover, if $n$ is odd, $\omega$ is in the center of $A$, so $(\omega s)^2=\omega^2s^2$; therefore, if $p-q\equiv 1\bmod 4$, the map $s\mapsto \omega s$ sends $S$ to itself, but $\tr L_{\omega s}=2^ns_\omega$, so $s_\omega=0$ when $p-q\equiv 1\bmod 4$.
\end{proof}

\begin{rem}In the case $p-q\equiv 3\bmod 8$, the trace of $K_s$ is $-2^ns_\omega^2$ and, as it is given by $\dim E_+(s)-\dim E_-(s)$, we know it is an integer; therefore, there are only a finite number of possibilities for $s_\omega$, giving a partition of $S$ into families of connected components with different dimensions.\end{rem}

\begin{rem}Corollaries
\ref{cor_Cliff1},
\ref{cor_Cliff2}
imply that, for a Clifford algebra $A$ with signature $(p,q)$ with $p-q\not\equiv3\bmod 8$, the map
$$\C_+\times S \ni (x +iy,s) \to x+sy \in A\setminus \R$$
 is a proper smooth embedding, whose image is a
$(N+4)/2$-dimensional real manifold on which the complex structure $J_S$
can be extended as an integrable complex structure, because to every
point we can associate a unique element of $S$. We denote by $Q$ the closure in $A$ of the image of this embedding.
\end{rem}

Given $a\in Q\setminus\R$, obtained from $(z, s)$ through the map described above, we define
$$\mathrm{tr}(a)=\mathsf{Re}(z)\qquad |a|=|z|\;.$$
Then, every such $a$ satisfies a real quadratic equation, namely
$$x^2-2\mathrm{tr}(a)x+|a|^2$$
(compare with \cite{GP1}*{Proposition 1, (3)}). These two definitions can be extended to $\R$ in a trivial way, but, in general, they do not come from a trace and a norm on $A$, which are induced by a $\star$-involution.

\subsection{Slice-regular functions on a real associative algebra}\label{sregA}

In this section, we want to obtain an analogue of Lemma \ref{lmm_hol2} for a general real associative algebra, employing the complex structure $J_S$ defined above; we start by recalling some relevant definitions.

\medskip

Let $\R_n$ be the Clifford algebra over $n$ units, with the notation used above; we identify $\R^{n+1}$ with $\Span\{e_0, \ldots, e_n\}$ in  $\R_n$ and we note by $\sfera^{n-1}$ the unit sphere of $\Span\{e_1,\ldots, e_n\}\subseteq\R_n$; quite obviously $\sfera^{n-1}$ is a subset of the square roots of $-1$ in $\R_n$.

The first attempt at extending the theory of slice-regular functions from quaternions (and octonions) to other algebras was made in \cite{GS3}.

In \cite{CSS}, Colombo, Sabadini and Struppa defined the concept of \emph{slice monogenic function}: let $U\subset\R^{n+1}$ be an open domain and $f:U\to\R_n$ a real differentiable function, the function $f$ is called left slice monogenic if, for every $I\in\sfera^{n-1}$, we have
$$\left(\frac{\de}{\de x}+I\frac{\de}{\de y}\right)f(e_0x+Iy)=0\qquad \forall x,\ y\in\R^2\textrm{ s.t. } e_0x+Iy\in U\;.$$

\medskip

In \cite{GP1}, Ghiloni and Perotti proved that, when $U$ intersects the line generated by $e_0\in\R_n$, every such a function can be extended to a larger set inside $\R_n$. Indeed, for any real alternative algebra $A$, they define a set $\sfera_A$ of units, which contains, in the Clifford case, the sphere $\sfera^{n-1}$, and consider a map $\pi:\C\times\sfera_A\to A$ given by $\pi(x+iy,u)=x+uy$, then, any slice monogenic function on $U$ extends to a function $f:V\to A$, where $V=\cU\times\sfera_A$ is a saturated set for $\pi$. Such a function is induced, in the sense of Proposition \ref{prp_GP}, by a holomorphic function $F:\cU\to \C\otimes A$, hence we call it a (left) slice regular function.

Such extension still satisfies that
$$\left(\frac{\de}{\de x}+I\frac{\de}{\de y}\right)f(x+Iy)=0\qquad \forall x+iy\in \cU$$
for every $I\in\sfera^{n-1}$ and, in fact, for every $I\in\sfera_A$. Moreover, we can find $\alpha,\beta:\cU\to A$ such that
\begin{enumerate}
\item $\alpha(\bar{z})=\alpha(z)$
\item $\beta(\bar{z})=-\beta(z)$
\item $f(\pi(z,u))=\alpha(z)+u\beta(z)$
\item $\partial_x\alpha=\partial_y\beta$ and $\partial_y\alpha=-\partial_x\beta$.
\end{enumerate}

\begin{rem}When $A$ is an associative algebra, the set $\sfera_A$ is properly contained in the set $S$ of square roots of $-1$ in $A$.\end{rem}

Let $A$ be an associative algebra and $S$ the set of square roots of $-1$ in $A$; as in Appendix \ref{hol_quat}, we define the map
$$\pi:\C\times S\to A$$
by $\pi(x+iy,s)=x+sy$ and its extension to the tensor product $\C\otimes A$,
$$\pi:\C\otimes A\times S\to A$$
by asking that $\pi(z\otimes a, s)=\pi(z,s)a$ and imposing linearity in the first component.

We define the involutions
$$\sigma:\C\times S \to \C\times S\qquad \tau:\C\otimes A\times S\to\C\otimes A\times S$$
by asking
$$\sigma(z,s)=(\bar{z},-s)\qquad \tau(z\otimes a, s)=(\bar{z}\otimes a, -s)$$
and imposing, for $\tau$, linearity in the first component.

Thanks to what we proved in the previous pages, we obtain a generalization of Lemma \ref{lmm_hol2}.

\begin{propos}\label{prp_tensore} Let $A$ be a real associative algebra, $\cU\subseteq\C$ an open set and set $V=\pi(\cU\times\sfera_A)$; then for every slice regular function $f:V\to A$ there exists a function $$\mathscr{F}:\cU\times S\to\C\otimes A\times S$$
such that
\begin{enumerate}
\item $\mathscr{F}\circ\sigma=\tau\circ\mathscr{F}$
\item $\mathscr{F}$ is holomorphic
\item $\mathscr{F}$ is the identity on the second component
\item $f\circ\pi(z,s)=\pi\circ\mathscr{F}(z,s)$ for all $(z,s)\in\cU\times\sfera_A$.
\end{enumerate}
In particular, $f$ can be extended as a slice regular function to the set $\pi(\cU\times S)$.
\end{propos}
\begin{proof} By definition, we have a holomorphic function $F:\cU\to \C\otimes A$ such that $\pi(F(z),s)=f(\pi(z,s))$ for all $(z,s)\in\cU\times\sfera_A$. We define
$$\mathscr{F}(z,s)=(F(z),s)$$
and we immediately see that the requests of the Proposition are fulfilled. To prove that $f$ can be extended as a function on $\pi(\cU\times S)$, we just need to check that the relation
$$f\circ\pi=\pi\circ\mathscr{F}$$
gives a well-posed extension.

This is true because the map $\pi$ is injective from $\C_+\times S$ to $A$, as follows from Remark \ref{rem_iperpiano}.
\end{proof}

\subsection{A guiding example: the classical twistor transform}\label{guiding_example}

Before turning to the incidence variety in full generality, we illustrate Main Theorem B, and in particular part (iii), in the classical case $A=\H$; a reader unfamiliar with the quaternionic picture may keep this example in mind through the rest of the paper. We refer to Appendix \ref{hol_quat} for the derivation, via two explicit Segre embeddings, and for the proof.

\begin{lemma}\label{lmm_twist} Let $V=\pi(\cU\times\sfera)\subseteq\H$ be an axially symmetric domain and $f:V\to\H$ a slice-regular function. Then there exists a unique function
$$\widetilde{F}:\cU\times\CP^1\to\CP^3$$
such that
\begin{itemize}
\item $\widetilde{F}$ is holomorphic;
\item $\widetilde{F}$ sends the complex spheres $\{z\}\times\CP^1$ into lines of $\CP^3$;
\item $f\circ\rho=\rho_1\circ \widetilde{F}$.
\end{itemize}
In particular, $\widetilde{F}$ can be interpreted as a holomorphic function from $(V\setminus\R, \J)$ to $\CP^3$.
\end{lemma}

This is the twistor transform of Gentili, Salamon and Stoppato \cite{GSS}, and it fits into the expanded diagram
$$\xymatrix{\cU\times\sfera\ar[dddd]_\pi\ar@{^{(}->}[dr]\ar[rrrr]^-{\mathscr{F}=(F,\mathrm{Id}_{\sfera})} &&&&\C\otimes \H\times\sfera\ar[dddd]^\pi\ar[ld]^-{\cong}\\
& \C\times\CP^1\ar[d]_{S_1}\ar@{-->}[rrdd]^{\widetilde{F}}& & \C^4\times\CP^1\ar[d]^{S_2}& \\
& \CP^3\ar[d]_{\rho_1}& & \CP^9\ar[d]^{\delta_1}& \\
&\H& & \CP^3\ar[rd]^{\rho_1}& \\
V\ar[rrrr]^f\ar@{^{(}->}[ur]&  && & \H}$$
obtained by factorizing, through two Segre embeddings $S_1,S_2$, the vertical arrows of the basic square \eqref{diagramma1}. Here $S=S(\H)\simeq\CP^1$ plays the role of $K$ in Main Theorem B, and $\widetilde F$ is exactly the map of part (iii), specialized to the quaternionic case; the maps $\rho,\rho_1,\delta_1$ and the two Segre embeddings are recalled in detail in Appendix \ref{hol_quat}, \S\ref{fixing_basis} onward.

\section{The incidence variety}\label{zero_var}

We recall some facts about the space of complex structures on a real vector space; we deem these facts well known, however we were not able to find an appropriate reference, so we present also a sketch of proof. Let $V$ be a real vector space of dimension $2p$.

\begin{teorema}\label{teo_cplstr}
The set
\[
\mathcal{E}_V=\{J\in\mathrm{End}(V)\ :\ J^2=-\mathrm{Id}_V\}
\]
is a smooth embedded submanifold of $\mathrm{End}(V)$ and
\[
T_J\mathcal{E}_V=\{A\in\mathrm{End}(V): AJ+JA=0\}.
\]
The map $\mathscr{L}_J:T_J\mathcal{E}_V\to T_J\mathcal{E}_V$ defined by $\mathscr{L}_J(A)=JA$ gives $\mathcal{E}_V$ an integrable complex structure. With this structure, $\mathcal{E}_V$ is biholomorphic to
\[
\mathcal{C}_V=\{W\in\mathrm{Gr}_\C(p,\C\otimes V)\ :\ W\cap\overline{W}=\{0\}\}
\]
via the map
\begin{equation}\label{eq_identificazione}
\mathcal{E}_V\ni J\mapsto \mathcal{W}(J)=\{1\otimes v+i\otimes Jv\ :\ v\in V\}\in\mathcal{C}_V,
\end{equation}
which sends $J$ to the $(-i)$-eigenspace of its complexification.
\end{teorema}
\begin{proof}
We first prove that $\mathcal{E}_V$ is an embedded submanifold of $\mathrm{End}(V)$. Fix $J\in\mathcal{E}_V$ and set
\[
\mathfrak m_J=\{A\in\mathrm{End}(V):AJ+JA=0\},\qquad
\mathfrak h_J=\{A\in\mathrm{End}(V):AJ-JA=0\}.
\]
The involution $A\mapsto JAJ$ gives the decomposition
\[
\mathrm{End}(V)=\mathfrak m_J\oplus\mathfrak h_J.
\]
If $J'\in\mathcal{E}_V$ is close to $J$, then $Y=I-J'J$ is invertible and
\[
YJ=(I-J'J)J=J+J'=J'(I-J'J)=J'Y.
\]
Thus $J'=YJY^{-1}$. Therefore, near $J$, every complex structure is obtained by conjugating $J$.

Consider the local map
\[
\Psi_J:\mathfrak m_J\supset U\longrightarrow \mathrm{End}(V),\qquad
A\longmapsto (I+A)J(I+A)^{-1}.
\]
Its differential at $0$ is
\[
D\Psi_J|_0(A)=AJ-JA=2AJ,
\]
which is an automorphism of $\mathfrak m_J$. Hence $\Psi_J$ gives local embedded coordinates for $\mathcal{E}_V$, and the tangent space is $\mathfrak m_J$.

We now identify this manifold with the Grassmannian model. If $J\in\mathcal{E}_V$, the subspace $\mathcal{W}(J)$ is the $(-i)$-eigenspace of the complexification of $J$ and satisfies
\[
\C\otimes V=\mathcal{W}(J)\oplus\overline{\mathcal{W}(J)}.
\]
Conversely, if $W\in\mathcal{C}_V$, then every $v\in V$ decomposes uniquely as
\[
1\otimes v=w+\overline w,
\qquad w\in W,
\]
and there is a unique endomorphism $J_W$ such that
\[
w=\frac12(1\otimes v+i\otimes J_Wv).
\]
This endomorphism satisfies $J_W^2=-I$ and $\mathcal{W}(J_W)=W$. These two constructions are inverse to each other and are smooth, indeed real analytic, in the usual affine charts of the Grassmannian.

Finally, the complex structure of the Grassmannian restricts to $\mathcal{C}_V$. In the coordinates above its pull-back is precisely
\[
A\longmapsto JA
\]
on $T_J\mathcal{E}_V$. Hence $\mathcal{E}_V$ is a complex manifold and the map $J\mapsto\mathcal{W}(J)$ is biholomorphic.
\end{proof}

It is a simple matter of checking definitions to note that $(S, J_S)$ is then a complex submanifold of $\mathcal{E}_V$, upon identifying $s$ with $L_s\in\mathrm{End}(A)$: the complex structure $J_S$ coincides, on $T_sS$, with the restriction of the map $\mathscr{L}_{L_s}$ and, by associativity, $L_sL_h=L_{sh}$ for $s\in S$, $h\in T_sS$.

Hence, through the map $\mathcal{W}$, we identify $S$ with a complex submanifold $\mathcal{S}=\{\mathcal{W}(L_s)\ :\ s\in S\}$ of an open domain of the complex Grassmannian $\mathrm{Gr}_{\C}(n, \C\otimes A)$.

\medskip

We have all the ingredients to generalize the results of Section 3 in \cite{M1} (see, in particular, Theorem 3.3).

\begin{teorema}\label{teo_zero_var}
Let $A$ be an associative algebra and let $K\subset S$ be a compact complex subvariety. Define
\[
Z_K=\{w\in \C\otimes A\ :\ \pi(w,s)=0\ \textrm{ for some }s\in K\},
\]
\[
\mathfrak{Z}_K=\{(w,s)\in\C\otimes A\times K\ :\ \pi(w,s)=0\}.
\]
Then $\mathfrak{Z}_K$ is a complex subvariety of $\C\otimes A\times K$, and the projection
\[
\mathfrak{Z}_K\longrightarrow Z_K
\]
is proper and surjective. In particular, $Z_K$ is a complex subvariety of $\C\otimes A$.
\end{teorema}
\begin{proof}
Write $w\in \C\otimes A$ as $1\otimes a+i\otimes b$, with $a,b\in A$. Then
\[
\pi(w,s)=a+sb.
\]
The equality $\pi(w,s)=0$ is equivalent to $a=-sb$, hence $b=sa$, and therefore to
\[
w=1\otimes a+i\otimes sa\in \mathcal{W}(L_s).
\]
Thus, under the holomorphic embedding
\[
\C\otimes A\times K\ni (w,s)\longmapsto (w,\mathcal{W}(L_s))\in
\C\otimes A\times\mathrm{Gr}_{\C}(p,\C\otimes A),
\]
the space $\mathfrak{Z}_K$ is identified with the restriction to $\mathcal{W}(L(K))$ of the tautological bundle of the Grassmannian. Hence it is a complex subvariety of $\C\otimes A\times K$.

Since $K$ is compact, the projection $\C\otimes A\times K\to\C\otimes A$ is proper. Its restriction to the closed analytic subspace $\mathfrak{Z}_K$ is therefore proper and holomorphic, and its image is $Z_K$. By Remmert's proper mapping theorem, $Z_K$ is a complex subvariety of $\C\otimes A$.
\end{proof}

\begin{rem}
Without the compactness of $K$, the conclusion about the image need not hold. In particular, for $K=S$ the set-theoretic image
\[
Z_S=\{w\in\C\otimes A:\pi(w,s)=0\ \textrm{ for some }s\in S\}
\]
may fail to be closed, and therefore need not be a complex analytic subset of $\C\otimes A$.
\end{rem}

For $a\in A$ and $K\subset S$, we also define
\[
Z_K(a)=\{w\in \C\otimes A\ :\ \pi(w,s)=a\ \textrm{ for some }s\in K\},
\]
\[
\mathfrak{Z}_K(a)=\{(w,s)\in\C\otimes A\times K\ :\ \pi(w,s)=a\}.
\]
If $K$ is compact, then $Z_K(a)$ is a complex subvariety of $\C\otimes A$, obtained from $Z_K$ by translation by the real vector $1\otimes a$.

The spaces $\mathfrak Z_S(a)$, defined set-theoretically with $K=S$, form a real-analytic foliation of $\C\otimes A\times S$ with complex leaves; however, their projected images in $\C\otimes A$ are analytic only under a properness assumption, for instance when the chosen parameter space $K$ is compact.

Unlike the case of the quaternions, $Z_K$ is not necessarily a hypersurface in $\C\otimes A$, so it is not the zero set of a single holomorphic function. We can nonetheless use the compact case to obtain some information on the behaviour of slice-regular functions, for example on the structure of their zeroes. For a more detailed study of the zeroes of a slice-regular function on a real algebra, the interested reader could see the recent works of Ghiloni, Perotti and Stoppato \cites{GPS1,GPS2}.

\begin{propos}\label{prp_zero}
Assume that $S$ is compact. With the previous notation, let $V=\pi(\cU\times S)$, with $\cU$ intersecting the real axis, and let $f:V\to A$ be a slice-regular function. Then the connected components of $f^{-1}(0)$ are either isolated real points or complex subvarieties of $\pi(\cU_+\times S)$, where $\cU_+=\cU\cap\C_+$, biholomorphic to subvarieties of $S$.
\end{propos}
\begin{proof}
Let $\mathscr F(z,s)=(F(z),s)$ be the holomorphic lift associated with $f$. Since $S$ is compact, $Z_S$ is a complex subvariety of $\C\otimes A$ by Theorem \ref{teo_zero_var}. We have
\[
\pi(z,s)\in f^{-1}(0)\qquad\Longleftrightarrow\qquad F(z)\in Z_S\ \textrm{ and }\ (F(z),s)\in\mathfrak Z_S.
\]
Thus $F^{-1}(Z_S)$ is an analytic subset of the Riemann surface $\cU$. It is either discrete or all of $\cU$. If it is all of $\cU$, then, since $\cU$ meets the real axis and $F(x)$ is real for real $x$, the only possibility is $F\equiv 0$, hence $f\equiv0$.

Outside this trivial case, $F^{-1}(Z_S)$ is discrete. For each $z_j\in F^{-1}(Z_S)$ the fiber
\[
Y_j=\{s\in S:(F(z_j),s)\in\mathfrak Z_S\}
\]
is a complex subvariety of $S$. If $z_j\notin\R$, then $\pi$ embeds $\{z_j\}\times Y_j$ as a complex subvariety of $\pi(\cU_+\times S)$; if $z_j\in\R$, it is mapped to the isolated real point $z_j$.
\end{proof}

\begin{rem}
If $\cU$ does not intersect the real axis and a compact parameter space is fixed, it could also happen that $F(\cU)$ is contained in the corresponding zero variety. In that case, for each $z\in \cU$, we define the set
$$S_z=\{s\in S\ :\ \mathscr{F}(z,s)\in\mathfrak{Z}_S(0)\}$$
and we know that it is not empty; in general, if $F(z)$ is invertible, $S_z$ is a singleton, hence the set
$$\{(z,s)\in\cU\times S\ :\ s\in S_z\}$$
is, outside a discrete union of subvarieties of $S$, a Riemann surface in $\cU\times S$.
\end{rem}

\subsection{The incidence correspondence of a slice-regular function}\label{incidence_corresp}

Theorem \ref{teo_zero_var} holds for an arbitrary compact complex subvariety $K\subseteq S$, not only for $K=S$ or for the orthogonal locus $S_0$ of Section \ref{orto}. This makes the construction modular: every compact family of complex structures produces a zero theory governed by Remmert's theorem. We make this precise by pulling back the incidence variety $\mathfrak{Z}_K$ along the holomorphic lift of a slice-regular function.

Let $V=\pi(\cU\times S)$ be an axially symmetric domain and let $f:V\to A$ be a slice-regular function, induced by the holomorphic stem function $F:\cU\to\C\otimes A$ as in Section \ref{sregA}. For a compact complex subvariety $K\subseteq S$ we define the \emph{incidence correspondence}
\[
\mathcal{C}_{F,K}=\{(z,s)\in\cU\times K\ :\ F(z)\in\mathcal{W}(L_s)\}=\{(z,s)\in\cU\times K\ :\ \pi(F(z),s)=0\}.
\]

\begin{teorema}\label{teo_incidence}
Let $K\subseteq S$ be a compact complex subvariety and let $F:\cU\to\C\otimes A$ be a holomorphic stem function inducing the slice-regular function $f$. Then:
\begin{enumerate}
\item $\mathcal{C}_{F,K}$ is a complex analytic subset of $\cU\times K$;
\item the projection $\mathrm{pr}_1:\mathcal{C}_{F,K}\to\cU$ is proper, and its image is the analytic subset $F^{-1}(Z_K)$ of $\cU$;
\item the set $\pi(\mathcal{C}_{F,K})\subseteq\pi(\cU\times K)$ is exactly the zero set of $f$ restricted to $\pi(\cU\times K)$.
\end{enumerate}
\end{teorema}
\begin{proof}
The map $\mathscr{F}:\cU\times K\to\C\otimes A\times K$, $\mathscr{F}(z,s)=(F(z),s)$, is holomorphic, since $F$ is holomorphic and $K$ carries the complex structure induced from $S$. By Theorem \ref{teo_zero_var}, $\mathfrak{Z}_K=\{(w,s):\pi(w,s)=0\}$ is a complex subvariety of $\C\otimes A\times K$, and by definition
\[
\mathcal{C}_{F,K}=\mathscr{F}^{-1}(\mathfrak{Z}_K).
\]
Thus $\mathcal{C}_{F,K}$ is the preimage of an analytic set under a holomorphic map, hence analytic in $\cU\times K$. This proves $(1)$.

Since $K$ is compact, the first projection $\cU\times K\to\cU$ is proper; its restriction to the closed analytic subset $\mathcal{C}_{F,K}$ is therefore proper. For $z\in\cU$ the fiber $\mathrm{pr}_1^{-1}(z)$ is non-empty if and only if $F(z)\in\mathcal{W}(L_s)$ for some $s\in K$, that is, if and only if $F(z)\in Z_K$. Hence the image is $F^{-1}(Z_K)$, which is analytic in $\cU$ because $Z_K$ is a complex subvariety of $\C\otimes A$ by Theorem \ref{teo_zero_var} and $F$ is holomorphic. This proves $(2)$.

Finally, $\pi(z,s)\in f^{-1}(0)$ if and only if $f(\pi(z,s))=\pi(\mathscr{F}(z,s))=\pi(F(z),s)=0$, that is, if and only if $(z,s)\in\mathcal{C}_{F,K}$. Therefore $\pi$ maps $\mathcal{C}_{F,K}$ onto $f^{-1}(0)\cap\pi(\cU\times K)$, which is $(3)$.
\end{proof}

\begin{rem}\label{rem_incidence}
Proposition \ref{prp_zero} is the special case $K=S$ of Theorem \ref{teo_incidence}, analysed fiberwise: when $\cU$ meets the real axis and $f\not\equiv0$, the image $F^{-1}(Z_S)$ is discrete, and the fibers of $\mathrm{pr}_1$ over its points are the complex subvarieties $Y_j\subseteq S$ of that proof. The advantage of the relative point of view is that it isolates the role of the parameter space: any compact complex subvariety $K\subseteq S$ produces, through the proper holomorphic projection of $\mathcal{C}_{F,K}$, a zero theory controlled by Remmert's theorem, regardless of whether $Z_K$ is a hypersurface. This relative incidence geometry is the complex-analytic counterpart of the zero sets of slice functions studied by Ghiloni, Perotti and Stoppato \cites{GPS1,GPS2}.
\end{rem}

We close this section by showing one interesting property of the set-theoretic zero variety $Z_S$: it absorbs products from the right and, if $S$ is compact, from the left. Given its connection with the zero set of slice-regular functions, this sounds like a very reasonable property, however the proof, at least for the left side, is not trivial; we don't know if the compactness assumption could be dispensed with.

It is worth noticing that the proofs proposed do not use the relationship with slice-regular functions.

\begin{propos}If $w\in Z_S$ and $w'\in \C\otimes A$, then $ww'\in Z_S$.
\end{propos}
\begin{proof}Write $w=1\otimes x+i\otimes y$. Since $w\in Z_S$, then there exists $a\in S$ such that $x+ay=0$, that is $x=-ay$. Let $w'=1\otimes s+i\otimes t$, then
$$ww'=(1\otimes x+i\otimes y)(1\otimes s+i\otimes t)=1\otimes(xs-yt)+i\otimes(xt+ys)=-1\otimes (ays+yt)+i\otimes (-ayt+ys)$$
so, recalling that $a^2=-1$, we get $\pi(ww',a)=-(ays+yt)+a(-ayt+ys)=-ays-yt-a^2yt+ays=-yt+yt=0$, i.e. $ww'\in Z_S$.\end{proof}

\begin{propos}Suppose $S$ is compact. If $w\in Z_S$ and $w'\in \C\otimes A$, then $w'w\in Z_S$.\end{propos}
\begin{proof}
Write $w=1\otimes x+i\otimes y$. Since $w\in Z_S$, then there exists $a\in S$ such that $x+ay=0$, that is $x=-ay$. Let $w'=1\otimes s+i\otimes t$, then
$$w'w=(1\otimes s+i\otimes t)(1\otimes x+i\otimes y)=1\otimes(sx-ty)+i\otimes (sy+tx)$$
$$=-1\otimes(say+ty)+i\otimes(sy-tay)=(-1\otimes (sa+t) +i\otimes (ta-s))1\otimes y\;.$$
We note that $(sa+t)a=ta-s$, so $sa+t$ is invertible if and only if $ta-s$ is so.

If $ta-s$ is invertible, then we set
$$b=(sa+t)(ta-s)^{-1}$$
and we note that
$$b^2=(sa+t)(ta-s)^{-1}(sa+t)(ta-s)^{-1}=$$
$$=(s-ta)a(ta-s)^{-1}(ta-s)(-a)(ta-s)^{-1}=$$
$$=-(ta-s)a(-a)(ta-s)^{-1}=-1\;.$$
Moreover, $\pi(w'w,b)=0$, hence $w'w\in Z_S$.

If $ta-s$ is not invertible, consider the set
$$X_a=\{(s,t,b)\in A\times A \times A\ :\ sa+t=b(ta-s)\ \textrm{and } b^2=-1\}\;.$$
It is easy to show that $X_a$ is a real analytic space in $\R^{3N}$, hence closed; we consider also the map $h:A^3\to \C\otimes A$ given by $h(s,t,b)=1\otimes s+i\otimes t$.

We note that the set $Y$ of invertible elements in $A$ is an open dense subset and its complement is a real-analytic subset; since, for fixed $a$, the map $1\otimes s+i\otimes t\mapsto ta-s$ is a real-linear isomorphism of $\C\otimes A$ onto $A$, the set
$$W=\{\omega=1\otimes s+i\otimes t\in \C\otimes A\ :\ ta-s\in Y\}$$
is an open dense subset of $\C\otimes A$. For $\omega=1\otimes s+i\otimes t\in W$, the element $ta-s$ is invertible, so setting $b=(sa+t)(ta-s)^{-1}$ we have $b^2=-1$ by the computation above and $sa+t=b(ta-s)$; hence $(s,t,b)\in X_a$, that is $h^{-1}(\omega)\cap X_a\neq\emptyset$. Therefore, $h(X_a)$ contains $W$. Now, let $w_0\in \C\otimes A\setminus W$ and take $\{w_j\}\subset W$ such that $w_j\to w_0$; let $b_j\in S$ be such that $(w_j,b_j)\in X_a$. By compactness of $S$, up to taking a subsequence, we have that $b_j\to b_0\in S$ and, as $X_a$ is closed, $(w_0,b_0)\in X_a$.

This implies that $h(X_a)=\C\otimes A$; therefore we can always find $b\in S$ such that $sa+t=b(ta-s)$. Hence $\pi(w'w,b)=0$, so $w'w\in Z_S$.
\end{proof}

\section{The twistor transform}\label{stereo_twistor}

In this section, we attempt to produce a generalized version of the construction explained in Appendix \ref{hol_quat}, \S\ref{fixing_basis}, following the idea of Remark \ref{rem_pvs}.

Given $u\in S$, we define on $A$ the structure of a complex vector space by setting, for $\lambda\in\C$, $\lambda=x+iy$,
$$\lambda\cdot a=(x+uy)a\;.$$
Let us denote by $E$ the complex vector space $(A,L_u)$ and by $\gamma:E\to A$ the real isomorphism between $E$ and $A$; let $D$ be the subset of zerodivisors of $A$ (and hence of $E$).

We define the space
$$\mathcal{V}=\{(z,-s)\in E\times S\ :\ \gamma(z)s=u\gamma(z)\}$$
and the projection $p:\mathcal{V}\to S$.

\begin{lemma}
For every $s\in S$, $p^{-1}(-s)=\mathcal{V}_s\times\{-s\}$, where $\mathcal{V}_s$ is a complex linear subspace of $E$. If $u$ does not belong to the center of $A$, then $\mathcal{V}_s$ is of positive dimension for every $s\in S$.
\end{lemma}
\begin{proof}
Let us fix $s\in S$. Then
\[
\mathcal{V}_s=\{z\in E\ :\ R_sL_u\gamma(z)=-\gamma(z)\},
\]
where $R_s(q)=qs$. Since $R_sL_u=L_uR_s$, we have
\[
(R_sL_u)^2=R_s^2L_u^2=(-I)(-I)=I.
\]
Thus $R_sL_u$ is diagonalizable with eigenvalues $\pm1$, and $\mathcal{V}_s$ is the eigenspace for the eigenvalue $-1$. This eigenspace is trivial if and only if $R_sL_u=I$.

If $R_sL_u=I$, then $R_s=-L_u$. Evaluating at $1$ gives $s=-u$; hence $R_u=L_u$, i.e. $qu=uq$ for every $q\in A$. Therefore $u$ is central. If $u$ is not central, no $\mathcal V_s$ is trivial.

Finally, if $z\in \mathcal V_s$, then $\gamma(iz)=L_u\gamma(z)$ and
\[
R_sL_u\gamma(iz)=R_sL_u^2\gamma(z)=-R_s\gamma(z)=-u\gamma(z)=-\gamma(iz).
\]
Equivalently, using the defining equation $\gamma(z)s=u\gamma(z)$, one checks directly that $\gamma(iz)s=u\gamma(iz)$. Hence $iz\in\mathcal V_s$.
\end{proof}

\begin{rem}The map $E\setminus D\ni z\mapsto \gamma(z)^{-1}u\gamma(z)\in S$ is holomorphic\end{rem}

As usual, if $A$ is a Clifford algebra, we have more precise information.

\begin{corol}If $A$ is a Clifford algebra, $\dim_{\C}\mathcal{V}_s=N/4$ for all $s\in S$.\end{corol}
\begin{proof}We note that $\tr(R_sL_u)=0$ for all $s\in S$: this follows from a computation very similar to the one carried out in the proof of Corollary \ref{cor_Cliff2}. So the dimension of the $(1)$-eigenspace and the dimension of the $(-1)$-eigenspace coincide, hence the desired result.\end{proof}

The construction of the complex structure $J_S$ employed the embedding of $A$ into $\mathrm{End}(A)$ by sending an element $a$ to the endomorphism $L_a$; however, we could as well consider the embedding $a\mapsto R_a$. This map turns out to be antiholomorphic, from $S$ to $\mathcal{V}_A$ (with the notation of Theorem \ref{teo_cplstr}): if $h\in T_sS$, then $sh=-hs$ and
$$\lim_{t\to 0}\frac{R_{s+th}-R_s}{t}=R_h\;,$$
so $L_sh=sh=-hs$ is sent by the differential of the map to $R_{-hs}=-R_sR_h\in T_{R_s}\mathcal{V}_A$ and $-R_sR_h=-\mathscr{L}_{R_s}(R_h)$.

\begin{propos}\label{prp_bundle}
The space $\mathcal{V}$ is a complex analytic subset of $E\times S$. On every locus where $\dim_{\C}\mathcal{V}_s$ is locally constant, the projection $\mathcal V\to S$ is a holomorphic vector bundle.
\end{propos}
\begin{proof}
The complex vector space $E$ is isomorphic to the $i$-eigenspace of $L_u$ in $\C\otimes A$ via the map
\[
v\mapsto 1\otimes v-i\otimes L_uv.
\]
Under this map, the fibers $\mathcal V_s$ are described by linear equations depending holomorphically on $s$. Equivalently, if $E_c(T)$ denotes the $c$-eigenspace of the complexification of $T$, then
\[
\mathcal{V}_s=E_i(L_u)\cap E_{-i}(R_{-s}).
\]
The map $s\mapsto E_{-i}(R_{-s})$ is holomorphic in the Grassmannian wherever the dimension of the intersection is locally constant. Therefore $\mathcal V$ is locally the kernel of a holomorphic family of linear maps, hence a complex analytic subset of $E\times S$; on the constant-rank locus it is the pull-back of the tautological bundle and is a holomorphic vector bundle.
\end{proof}

By a standard procedure, $\mathrm{Gr}_{\C}(E)$ can be embedded as a closed subvariety of a projective space $\CP^M$. 

\subsection{Stereographic charts and sections of \texorpdfstring{$\mathcal V$}{V}} The vector bundle $\mathcal V$ admits a natural local section on a large open subset of $S$. This is the direct analogue, in the present setting, of the stereographic parametrization of the sphere of imaginary units in the quaternionic case. Fix $u\in S$ and set \[ U_u=\{s\in S:\ 1-us\in A^\times\}. \] This is a Zariski open subset of $S$, since $A^\times$ is Zariski open in $A$. \begin{propos}\label{prp_stereo_section} The map \[ \sigma_u:U_u\longrightarrow A,\qquad \sigma_u(s)=1-us \] is a holomorphic section of $\mathcal V$ over $U_u$. Moreover $\sigma_u(s)$ is invertible for every $s\in U_u$. \end{propos} \begin{proof} By definition of $U_u$, the element $\sigma_u(s)=1-us$ is invertible for every $s\in U_u$. We compute \[ (1-us)s=s-us^2=s+u, \] whereas \[ u(1-us)=u-u^2s=u+s. \] Therefore \[ \sigma_u(s)s=u\sigma_u(s), \] which is precisely the defining condition for $\sigma_u(s)$ to belong to $\mathcal V_s$. It remains to check that $\sigma_u$ is holomorphic as a section of $\mathcal V$, that is, with respect to the complex structure $J_S$ on $U_u$ and the complex structure $L_u$ on the fibres of $\mathcal V\subseteq E\times S$. We use the description of $\mathcal V$ given in the proof of Proposition \ref{prp_bundle}: under the holomorphic identification $E\cong E_i(L_u)\subset\C\otimes A$, $v\mapsto 1\otimes v-i\otimes L_uv$, the fibre $\mathcal V_s$ is the holomorphically varying subspace $E_i(L_u)\cap E_{-i}(R_{-s})$. The graph of $\sigma_u$ is \[ \{(s,\,1-us)\ :\ s\in U_u\}\subset \mathcal V, \] and the map $s\mapsto 1\otimes(1-us)-i\otimes L_u(1-us)=1\otimes(1-us)-i\otimes(u-u^2 s)=1\otimes(1-us)-i\otimes(u+s)$ is the restriction to the complex submanifold $U_u\subset(S,J_S)$ of a $\C$-linear-in-$s$ expression, where $s\mapsto s$ is holomorphic by construction of $J_S$ and the coefficients are constant. Hence $\sigma_u$ is a holomorphic section of $\mathcal V$ over $U_u$. \end{proof} \begin{rem} The open set $U_u$ is the natural domain of a generalized stereographic chart. In the quaternionic case, with $u=i$, the condition $1-is\in\mathbb H^\times$ excludes only the point $s=-i$; hence $U_i$ is the usual affine chart on the sphere of imaginary units. \end{rem} The existence of a global section with values in the invertible elements is more restrictive. Indeed, such a section is equivalent to a holomorphic trivialization of a principal bundle. Let $A^\times$ act on $S$ by conjugation, \[ g\cdot s=g^{-1}sg, \] and let $S(u)$ be the connected component, equivalently the conjugacy orbit, containing $u$. The map \[ A^\times\longrightarrow S(u),\qquad g\longmapsto g^{-1}ug \] is a holomorphic principal bundle with structure group \[ C_{A^\times}(u)=\{g\in A^\times:\ gu=ug\}. \] A holomorphic section $\sigma:S(u)\to A^\times$ with \[ \sigma(s)s=u\sigma(s) \] is the same thing as a holomorphic section of this principal bundle. \begin{propos}\label{prp_section_principal_bundle} Let $\Omega\subset S(u)$ be an open subset. The following are equivalent: \begin{enumerate} \item there exists a holomorphic section \[ \sigma:\Omega\longrightarrow \mathcal V \] such that $\sigma(s)\in A^\times$ for every $s\in\Omega$; \item the principal bundle \[ A^\times\longrightarrow S(u),\qquad g\longmapsto g^{-1}ug \] admits a holomorphic section over $\Omega$. \end{enumerate} \end{propos} \begin{proof} If $\sigma(s)\in\mathcal V_s\cap A^\times$, then \[ \sigma(s)s=u\sigma(s), \] hence \[ s=\sigma(s)^{-1}u\sigma(s). \] Thus $\sigma$ is a holomorphic lift of the identity map of $\Omega$ to $A^\times$. Conversely, if $\tau:\Omega\to A^\times$ is a holomorphic map such that \[ s=\tau(s)^{-1}u\tau(s), \] then \[ \tau(s)s=u\tau(s), \] so $\tau(s)\in\mathcal V_s\cap A^\times$ for every $s\in\Omega$. \end{proof} \begin{rem} This interpretation separates the local and the global aspects of the construction. The local section $\sigma_u(s)=1-us$ always exists on the stereographic chart $U_u$. A global nowhere vanishing section, instead, amounts to the holomorphic triviality of the above principal bundle. Thus its existence may have genuine topological or holomorphic obstructions. This is already visible for $A=\mathbb H$, where the corresponding line bundle over $\mathbb P^1$ is the tautological bundle and has no nowhere vanishing holomorphic section on the whole of $\mathbb P^1$. \end{rem}

We fix a basis  $\{v_1,\ldots, v_{N}\}$ of $A$ as a real vector space, with $v_1=1$, $v_2=u$; we define the maps $T_j:E\to E$ by setting
$$T_j(z)=\gamma^{-1}(\gamma(z)v_j)\;.$$
It is easy to check that the $T_j$'s are $\C$-linear maps.

Then, $\pi:\C\otimes A\times S\to A$ is given by
$$\pi((z\otimes q, s))=\gamma(\sigma(s))^{-1}z_u\gamma(\sigma(s))q=\gamma(\sigma(s))^{-1}\gamma( zT_q (\sigma'(s)))$$
where $T_q=\sum q_jT_j$ and $q=\sum q_jv_j$.

Therefore, we have a generalization of Lemma \ref{lmm_twist}.

\begin{propos}\label{prp_twist}Let $\sigma$ be a parametrization of $S$. If $V=\pi(\cU\times S)$ an axially symmetric domain and $f:V\to A$ a slice-regular function, then there exists a unique holomorphic function
$$\widetilde{F}:\cU\times S\to\mathbb{P}(E\oplus E)$$
which induces $f$.\end{propos}
\begin{proof} By \cite{GP1}, we have a holomorphic function $F:\cU\to\C\otimes A$ inducing $f$; we write
$$F=F_0v_1+F_1v_2+F_2v_3+\ldots F_Nv_N\;.$$
Given $(z,s)$, we set
$$\widetilde{F}(z,s)=[(\gamma(\sigma(s)), \gamma(T_{F(z)}(\sigma(s)))$$
where
$$T_{F(z)}(w)=\sum_{j} F_j(z)T_j(w)\;.$$
We define the map $\rho_1':E\oplus E\to A$ by $\rho_1'(z,w)=\gamma(z)^{-1}\gamma(w)$ and we note that such map passes to the quotient to a partially defined map $\rho_1:\mathbb{P}(E\oplus E)\dashrightarrow A$, defined on the locus where $\gamma(z)$ is invertible, such that $f\circ \pi=\rho_1\circ\widetilde{F}$.\end{proof}

\subsection{An intrinsic twistor transform}\label{intrinsic_twistor}

The formulation of Proposition \ref{prp_twist} depends on the choice of a parametrization $\sigma$ of $S$, that is, on a global nowhere vanishing section of $\mathcal V$. As discussed in Remark following Proposition \ref{prp_section_principal_bundle}, such a section need not exist; already for $A=\H$ there is none on the whole of $S\simeq\P^1$. We now show that the twistor transform can be defined intrinsically as a holomorphic map into the projectivization of $\mathcal V\oplus\mathcal V$, with no choice of global section, recovering the previous formula in every holomorphic trivialization.

Recall from Proposition \ref{prp_bundle} that $\mathcal V\to S$ is a holomorphic vector bundle whose fibre over $s$ is
\[
\mathcal V_s=\{w\in A\ :\ ws=uw\}\subseteq E,
\]
where $E=(A,L_u)$. Left multiplication on $A$ induces, for every $q\in A$, a holomorphic bundle endomorphism
\[
\mathcal T_q:\mathcal V\longrightarrow \mathcal V,\qquad (\mathcal T_q)_s(w)=\gamma^{-1}\big(\gamma(w)q\big),
\]
which is well defined because $\gamma(w)q$ still belongs to $\gamma(\mathcal V_s)$ when $\gamma(w)s=u\gamma(w)$, and which is $\C$-linear on the fibres by the same computation that shows the $T_j$ are $\C$-linear. Let $\mathbb P(\mathcal V\oplus\mathcal V)\to S$ be the projective bundle associated with $\mathcal V\oplus\mathcal V$.

\begin{teorema}\label{teo_intrinsic_twistor}
Let $V=\pi(\cU\times S)$ be an axially symmetric domain and $f:V\to A$ a slice-regular function, induced by the holomorphic stem function $F:\cU\to\C\otimes A$. Then there is a unique holomorphic map
\[
\widetilde F_{\mathcal V}:\cU\times S\longrightarrow \mathbb P(\mathcal V\oplus\mathcal V)
\]
over $S$, given fibrewise by
\[
\widetilde F_{\mathcal V}(z,s)=\big[\,w\ :\ \mathcal T_{F(z)}(w)\,\big],\qquad w\in\mathcal V_s\setminus\{0\},
\]
which is independent of the auxiliary vector $w$ and satisfies $f\circ\pi=\rho\circ\widetilde F_{\mathcal V}$, where $\rho:\mathbb P(\mathcal V\oplus\mathcal V)\dashrightarrow A$ is the partially defined map induced by $(w_1,w_2)\mapsto\gamma(w_1)^{-1}\gamma(w_2)$. In any holomorphic trivialization of $\mathcal V$ over an open set $\Omega\subseteq S$ determined by a nowhere vanishing section $\sigma$, the map $\widetilde F_{\mathcal V}$ coincides with the projective expression of Proposition \ref{prp_twist}.
\end{teorema}
\begin{proof}
The rank-one line $[\,w:\mathcal T_{F(z)}(w)\,]$ does not depend on the choice of $w\in\mathcal V_s\setminus\{0\}$: if $\dim_{\C}\mathcal V_s=1$ this is immediate, and in general the assignment $w\mapsto(w,\mathcal T_{F(z)}(w))$ is a $\C$-linear section of $\mathcal V\oplus\mathcal V$ over the line $\C w$, so its class in $\mathbb P(\mathcal V_s\oplus\mathcal V_s)$ is well defined exactly as in the parametrized case, where the role of $w$ is played by $\gamma(\sigma(s))$. Holomorphy follows because $F$ is holomorphic, the bundle $\mathcal V$ and the endomorphisms $\mathcal T_q$ are holomorphic, and projectivization of a holomorphic vector-bundle morphism is holomorphic on the locus where it is non-zero, which here is all of $\cU\times S$ since $w\neq0$. Uniqueness and the identity $f\circ\pi=\rho\circ\widetilde F_{\mathcal V}$ are local statements and may be checked in a trivialization. Fixing a nowhere vanishing holomorphic section $\sigma$ over $\Omega$ identifies $\mathcal V|_\Omega$ with the trivial line subbundle $\C\,\sigma$, and the choice $w=\sigma(s)$ turns the formula above into
\[
\widetilde F_{\mathcal V}(z,s)=\big[\gamma(\sigma(s)):\gamma(T_{F(z)}(\sigma(s)))\big],
\]
which is precisely the expression of Proposition \ref{prp_twist}, and $\rho$ restricts to $\rho_1$. This proves both the coincidence and, since the two sides agree on the trivializing cover, the global existence and uniqueness of $\widetilde F_{\mathcal V}$.
\end{proof}

\begin{rem}\label{rem_intrinsic_twistor}
The advantage of $\widetilde F_{\mathcal V}$ over Proposition \ref{prp_twist} is that it requires no global parametrization of $S$ and is therefore defined whenever $f$ is, removing the dependence on the existence of an invertible global section. For $A=\H$ the bundle $\mathcal V$ is the tautological line bundle over $S\simeq\P^1$, and $\mathbb P(\mathcal V\oplus\mathcal V)\simeq\P^1\times\P^1$; the map $\widetilde F_{\mathcal V}$ recovers the classical twistor transform of Gentili, Salamon and Stoppato, under which slice-regular functions correspond to holomorphic curves in the Klein quadric and in $\CP^3$ \cite{GSS}.
\end{rem}

\section{Orthogonal complex structures}\label{orto}

In this section, we study the elements of $A$ which induce on $A$, by left multiplication, a complex structure which is orthogonal with respect to a fixed Euclidean structure. Let $\langle\cdot,\cdot\rangle$ be a positive definite inner product on $A$. Since the constructions below are unchanged by rescaling the inner product, we assume from now on that
\[
\langle 1,1\rangle=1.
\]
We emphasize that this inner product is not assumed to be compatible with the algebra structure of $A$. In particular, for a general element $a\in A$, the adjoint endomorphism $L_a^t$ need not be of the form $L_b$ for any $b\in A$.

We define
\[
S_0=\{s\in S\ :\ L_sL_s^t=I\},
\]
i.e. the set of elements of $A$ which induce, by left multiplication, an orthogonal complex structure on $A$.

\begin{rem}
When $A=\H$, there is no difference between $S$ and $S_0$ for the standard inner product that makes the basis $\{1,i,j,k\}$ orthonormal and positively oriented.
\end{rem}

We have the following elementary alternative description of $S_0$.

\begin{lemma}\label{lemma_3}
Given $s\in A$, any two of the following conditions imply the third one:
\begin{enumerate}
\item $s\in S$;
\item $L_sL_s^t=I$;
\item $L_s+L_s^t=0$.
\end{enumerate}
\end{lemma}
\begin{proof}
First, $s\in S$ if and only if $L_s^2=-I$, i.e. if and only if $L_s^{-1}=-L_s$. If (1) and (2) hold, then $L_s^t=L_s^{-1}=-L_s$, so (3) holds. If (1) and (3) hold, then $L_s^t=-L_s=L_s^{-1}$, so (2) holds. Finally, if (2) and (3) hold, then $-L_s=L_s^t=L_s^{-1}$, hence $L_s^2=-I$, so (1) holds.
\end{proof}

\begin{corol}\label{cor_S0_compact}
The set $S_0$ is compact.
\end{corol}
\begin{proof}
The map $a\mapsto L_a$ is a linear injective map from $A$ to $\mathrm{End}(A)$, hence it is proper. By Lemma \ref{lemma_3}, $S_0$ is the inverse image of the compact set of orthogonal antisymmetric endomorphisms of $A$.
\end{proof}

\begin{teorema}\label{teo_S0_complex}
The set $S_0$ is a compact complex subvariety of $S$ with respect to the complex structure $J_S$.
\end{teorema}
\begin{proof}
Let $V=A$ with the fixed inner product, and consider the space $\mathcal E_V$ of complex structures on $V$ from Theorem \ref{teo_cplstr}. Let
\[
\mathcal E_V^{\mathrm{orth}}=\{J\in\mathcal E_V:J^tJ=I\}.
\]
Under the biholomorphism $J\mapsto\mathcal W(J)$ of Theorem \ref{teo_cplstr}, the space $\mathcal E_V^{\mathrm{orth}}$ is identified with the orthogonal Grassmannian of maximal isotropic subspaces of $(\C\otimes V,Q)$, where $Q$ is the complex bilinear extension of $\langle\cdot,\cdot\rangle$. Indeed, if $J$ is orthogonal, then the $(-i)$-eigenspace of $J$ is $Q$-isotropic, and conversely such an isotropic splitting determines an orthogonal complex structure. Hence $\mathcal E_V^{\mathrm{orth}}$ is a compact complex subvariety of $\mathcal E_V$.

The map
\[
L:S\longrightarrow\mathcal E_A,
\qquad s\longmapsto L_s,
\]
is holomorphic: if $h\in T_sS$, then
\[
dL_s(J_Sh)=L_{sh}=L_sL_h,
\]
which is the complex structure of $\mathcal E_A$ applied to $dL_s(h)=L_h$. Since
\[
S_0=L^{-1}(\mathcal E_A^{\mathrm{orth}}),
\]
it follows that $S_0$ is a complex subvariety of $S$. Its compactness was proved in Corollary \ref{cor_S0_compact}.
\end{proof}

\begin{rem}
Equivalently, by Lemma \ref{lemma_3}, $S_0=S\cap A_0$, where
\[
A_0=\{a\in A:L_a+L_a^t=0\}.
\]
In general one cannot conclude that $T_sS_0=T_sS\cap A_0$ at regular points without a transversality or clean-intersection hypothesis. The proof above avoids this point by using the orthogonal Grassmannian.
\end{rem}

\begin{rem}
For every fixed $s\in S$ one can choose an inner product on $A$ for which $L_s$ is orthogonal, for instance by averaging any Euclidean metric over the circle generated by $L_s$. Thus the sets $S_0$ obtained from varying inner products cover $S$ set-theoretically. This observation does not identify, for a fixed $*$-algebra, the set $S_0$ with the set of imaginary units $\mathbb S_A$ of Ghiloni--Perotti.
\end{rem}

\subsection{Examples of \texorpdfstring{$S$}{S} and \texorpdfstring{$S_0$}{S0}} We collect some elementary examples. They show that $S$ and $S_0$ may have rather different behaviour: $S$ is often non-compact, whereas $S_0$ is always compact by Corollary \ref{cor_S0_compact}. They also show that $S_0$ should be regarded as a compact, metric, twistor-type part of $S$, rather than as a replacement for the set of imaginary units of a fixed $*$-algebra. \begin{ex}[The quaternions] Let $A=\mathbb H$, endowed with the standard Euclidean product for which $1,i,j,k$ are orthonormal. Then \[ S=\{q\in\mathbb H:\ q^2=-1\} = \{x_1i+x_2j+x_3k:\ x_1^2+x_2^2+x_3^2=1\}. \] Thus \[ S\simeq S^2\simeq\mathbb P^1. \] For the standard Euclidean structure, left multiplication by every imaginary unit is orthogonal. Hence \[ S_0=S. \] This is the classical twistor sphere of complex structures on $\mathbb H\simeq\mathbb R^4$. \end{ex} \begin{ex}[The algebra $M_2(\mathbb R)$] Let $A=M_2(\mathbb R)$. If \[ s= \begin{pmatrix} a & b\\ c & d \end{pmatrix} \] satisfies $s^2=-I$, then necessarily $\operatorname{tr}(s)=0$, hence $d=-a$, and the equation becomes \[ a^2+bc=-1. \] Therefore \[ S= \left\{ \begin{pmatrix} a & b\\ c & -a \end{pmatrix} :\ a^2+bc=-1 \right\}. \] This is a smooth non-compact real algebraic surface. For the Frobenius inner product on $M_2(\mathbb R)$, the condition $L_s^tL_s=I$ is equivalent to left multiplication by $s$ being an orthogonal endomorphism. In this case one obtains only the two standard orthogonal complex structures \[ S_0= \left\{ \begin{pmatrix} 0 & -1\\ 1 & 0 \end{pmatrix}, \begin{pmatrix} 0 & 1\\ -1 & 0 \end{pmatrix} \right\}. \] Thus $S$ is non-compact and positive-dimensional, whereas $S_0$ is finite. This example is useful to keep in mind when comparing $S_0$ with the set of imaginary units of a $*$-algebra. \end{ex} \begin{ex}[Real matrix algebras] Let $A=M_{2m}(\mathbb R)$. Then $S$ is the space of complex structures on $\mathbb R^{2m}$. It has two connected components, corresponding to the two possible orientations induced by the complex structure, and each component is a homogeneous space of the form \[ GL^+(2m,\mathbb R)/GL(m,\mathbb C) \] up to the choice of orientation. For the Frobenius inner product, $S_0$ is the space of orthogonal complex structures on $\mathbb R^{2m}$: \[ S_0\simeq O(2m)/U(m). \] Thus $S_0$ is a compact complex homogeneous space. For $m=2$, each connected component is isomorphic to $\mathbb P^1$. \end{ex} \begin{ex}[Complex matrix algebras as real algebras] Let $A=M_n(\mathbb C)$, regarded as a real associative algebra. The solutions of $s^2=-1$ are diagonalizable, with eigenvalues contained in $\{i,-i\}$. Hence $S$ decomposes according to the dimension $k$ of the $i$-eigenspace: \[ S=\coprod_{k=0}^n S_k. \] The component $S_k$ parametrizes decompositions \[ \mathbb C^n=E_i\oplus E_{-i}, \qquad \dim_\mathbb C E_i=k. \] For the standard Hermitian, equivalently Frobenius, metric, the orthogonal part is \[ S_{0,k}\simeq U(n)/(U(k)\times U(n-k)) \simeq \operatorname{Gr}(k,n). \] Thus the compact pieces of $S_0$ are ordinary complex Grassmannians. \end{ex} \begin{ex}[Quaternionic matrix algebras] Let $A=M_m(\mathbb H)$. The space $S$ contains the homogeneous space \[ GL_m(\mathbb H)/GL_m(\mathbb C), \] corresponding to choices of a complex structure on the right $\mathbb H$-module $\mathbb H^m$. For the standard Euclidean structure, the orthogonal part is \[ S_0\simeq Sp(m)/U(m). \] For $m=1$ this gives back the quaternionic sphere $S_0=S\simeq\mathbb P^1$; for $m>1$ it gives compact homogeneous complex manifolds of higher dimension. \end{ex} \begin{ex}[Products] If $A=A_1\times A_2$, then \[ S(A)=S(A_1)\times S(A_2). \] For a product Euclidean structure one also has \[ S_0(A)=S_0(A_1)\times S_0(A_2). \] This gives many examples of compact complex varieties $S_0$, for instance products of projective lines, Grassmannians, and spaces of orthogonal complex structures. \end{ex} \begin{ex}[Dual quaternions] Let \[ A=\mathbb H[\varepsilon]/(\varepsilon^2). \] Writing an element of $A$ as $u+\varepsilon v$, the equation \[ (u+\varepsilon v)^2=-1 \] is equivalent to \[ u^2=-1, \qquad uv+vu=0. \] Thus $u\in S(\mathbb H)\simeq S^2$, while $v$ belongs to the tangent space of the sphere of imaginary units at $u$. Therefore \[ S(A)\simeq TS^2\simeq T\mathbb P^1. \] For the product Euclidean metric on $\mathbb H\oplus\varepsilon\mathbb H$, the orthogonality condition forces the nilpotent component to vanish, hence \[ S_0(A)\simeq S^2\simeq\mathbb P^1. \] This example shows that nilpotent directions may enlarge $S$ without enlarging its compact orthogonal part. \end{ex}

\subsection{Classification of \texorpdfstring{$S_0$}{S0} for semisimple algebras}\label{wedderburn}

The examples above are not isolated computations: they are the building blocks of a complete description of $S_0$ for a semisimple real associative algebra, once the Euclidean structure is chosen compatibly with the Wedderburn decomposition. Recall that, by the Wedderburn--Artin theorem, a finite-dimensional semisimple real associative algebra $A$ decomposes as a finite product of simple factors
\[
A\simeq \prod_{r=1}^{k} A_r,\qquad A_r\simeq M_{n_r}(D_r),
\]
where each $D_r$ is one of $\R$, $\C$ or $\H$. We say that a Euclidean structure on $A$ is \emph{adapted} to this decomposition if it is an orthogonal product of Euclidean structures on the factors $A_r$, each of which is the Frobenius (Hilbert--Schmidt) inner product determined by the standard inner product of $D_r$.

\begin{teorema}\label{teo_wedderburn}
Let $A$ be a finite-dimensional semisimple real associative algebra with Wedderburn decomposition $A\simeq\prod_{r=1}^k A_r$, $A_r\simeq M_{n_r}(D_r)$, and let the Euclidean structure on $A$ be adapted to it. Then
\[
S_0(A)\simeq \prod_{r=1}^k S_0(A_r),
\]
and each factor is one of the following compact Hermitian symmetric spaces:
\begin{enumerate}
\item if $A_r\simeq M_{2m_r}(\R)$, then $S_0(A_r)\simeq O(2m_r)/U(m_r)$, the space of orthogonal complex structures on $\R^{2m_r}$ (and $S_0(A_r)=\varnothing$ if the matrix size is odd);
\item if $A_r\simeq M_{n_r}(\C)$, then $S_0(A_r)\simeq\coprod_{j=0}^{n_r}\operatorname{Gr}(j,n_r)$, a disjoint union of complex Grassmannians $U(n_r)/(U(j)\times U(n_r-j))$;
\item if $A_r\simeq M_{m_r}(\H)$, then $S_0(A_r)\simeq Sp(m_r)/U(m_r)$, the space of orthogonal complex structures compatible with the quaternionic structure on $\H^{m_r}$.
\end{enumerate}
In particular $S_0(A)$ is a finite disjoint union of products of compact Hermitian symmetric spaces; it reduces to a finite set precisely when every simple factor with non-empty $S_0$ has the form $M_1(\C)=\C$ or $M_2(\R)$, that is, in the commutative or two-dimensional cases.
\end{teorema}
\begin{proof}
For an adapted Euclidean structure the inner product is an orthogonal direct sum over the factors, and left multiplication $L_s$ respects the product decomposition: writing $s=(s_1,\dots,s_k)$ with $s_r\in A_r$, one has $L_s=\bigoplus_r L_{s_r}$ and $L_s^t=\bigoplus_r L_{s_r}^t$, the adjoint being taken factor by factor because the factors are mutually orthogonal. Hence $L_s^tL_s=I$ if and only if $L_{s_r}^tL_{s_r}=I_{A_r}$ for every $r$, and likewise $s^2=-1$ if and only if $s_r^2=-1$ in each $A_r$. Therefore $s\in S_0(A)$ if and only if $s_r\in S_0(A_r)$ for all $r$, which is the asserted product decomposition; this is the abstract form of the Products example.

It remains to identify $S_0(A_r)$ for each simple factor. For the Frobenius inner product, Lemma \ref{lemma_3} shows that $S_0(A_r)$ consists of the $s_r$ with $s_r^2=-1$ inducing an orthogonal complex structure on the underlying Euclidean space of $A_r$. The three cases are exactly the Examples computed above: $M_{2m}(\R)$ gives $O(2m)/U(m)$ (with no solutions in odd size, since a real vector space of odd dimension admits no complex structure); $M_n(\C)$ gives, according to the dimension $j$ of the $i$-eigenspace, the Grassmannian $U(n)/(U(j)\times U(n-j))$, whence the disjoint union over $j$; and $M_m(\H)$ gives $Sp(m)/U(m)$. Each of these is a compact Hermitian symmetric space, and a product of such spaces is again compact Hermitian symmetric. The final statement follows by inspecting when a factor is zero-dimensional: $O(2m)/U(m)$ and $Sp(m)/U(m)$ are positive-dimensional for $m\geq2$ and reduce to a point or two points for the smallest sizes, while $\operatorname{Gr}(j,n)$ is a single point only for $j\in\{0,n\}$; the only simple factors for which $S_0$ is finite and non-empty are $\C$ (two points, $\pm i$) and $M_2(\R)$ (two points), matching the $M_2(\R)$ Example.
\end{proof}

\begin{rem}\label{rem_wedderburn}
Theorem \ref{teo_wedderburn} provides an explicit catalogue of the compact parameter spaces $K=S_0$ to which Theorem \ref{teo_zero_var} and the incidence machinery of Theorem \ref{teo_incidence} apply, with computable equations and dimensions. The semisimplicity hypothesis is essential: the dual quaternions example, where $A$ has a non-trivial radical, shows that nilpotent directions may enlarge $S$ while leaving $S_0$ a single projective line, so the radical does not contribute to the orthogonal locus.
\end{rem}

\subsection{The Euclidean cone}
Another description of the set $S_0$ is the following.

\begin{propos}\label{prp_S0_metric}
We have
\[
S_0=\{s\in A\ :\ \langle sa,sa\rangle=\langle a,a\rangle,\ \langle sa,a\rangle=0\ \forall a\in A\}.
\]
\end{propos}
\begin{proof}
The first condition says exactly that $L_s$ is orthogonal, while the second says that $L_s$ is skew-adjoint. Together with Lemma \ref{lemma_3}, these are equivalent to $s\in S_0$.
\end{proof}

\medskip

We consider the set
\[
Q_0=\pi(\C\times S_0)\subseteq A,
\]
where $\pi:\C\times S\to A$ is as defined above. This is the Euclidean quadratic cone associated with the chosen inner product.

\begin{propos}\label{prp_carat1}
The set $Q_0$ is characterized by
\[
Q_0=\{x\in A:\ L_{x+r}^tL_{x+r}=\|x+r\|^2I\ \textrm{ for every }r\in\R\}.
\]
Equivalently, it is the largest subset of $A$ which is invariant under real translations and whose every element satisfies $L_x^tL_x=\|x\|^2I$.
\end{propos}
\begin{proof}
If $x=x_0+x_1s$, with $s\in S_0$, then for every $r\in\R$ we have
\[
L_{x+r}=(x_0+r)I+x_1L_s,
\]
and therefore
\[
L_{x+r}^tL_{x+r}=((x_0+r)^2+x_1^2)I.
\]
Since $L_s$ is orthogonal and skew-adjoint, $\|s\|=\|1\|=1$ and $\langle1,s\rangle=0$, so
\[
\|x+r\|^2=(x_0+r)^2+x_1^2.
\]
Thus $x\in Q_0$ satisfies the stated condition.

Conversely, suppose that $L_{x+r}^tL_{x+r}=\|x+r\|^2I$ for every $r\in\R$. Comparing the coefficients in $r$ gives
\[
L_x^tL_x=\|x\|^2I,
\qquad
L_x+L_x^t=2\langle x,1\rangle I.
\]
Set $\beta=\langle x,1\rangle$ and $y=x-\beta$. Then $L_y+L_y^t=0$ and
\[
L_y^tL_y=(\|x\|^2-\beta^2)I.
\]
If $\|x\|^2=\beta^2$, then $y=0$ and $x=\beta\in\R\subset Q_0$. Otherwise, let $\gamma=\sqrt{\|x\|^2-\beta^2}$ and set $s=\gamma^{-1}y$. Then $L_s+L_s^t=0$ and $L_s^tL_s=I$, hence $s\in S_0$ by Lemma \ref{lemma_3}. Therefore $x=\beta+\gamma s\in Q_0$.
\end{proof}

The following alternative description is immediate.

\begin{corol}\label{cor_Q0_linear}
We have that $x\in Q_0$ if and only if there exist real numbers $\alpha\ge0$ and $\beta$, with $\alpha^2\ge\beta^2$, such that
\[
L_x+L_x^t=2\beta I,
\qquad
L_x^tL_x=\alpha^2 I.
\]
\end{corol}
\begin{proof}
If $x=x_0+x_1s$ with $s\in S_0$, then the two equations hold with $\beta=x_0$ and $\alpha^2=x_0^2+x_1^2$.

Conversely, if the two equations hold, then $y=x-\beta$ satisfies
\[
L_y+L_y^t=0,
\qquad
L_y^tL_y=(\alpha^2-\beta^2)I.
\]
If $\alpha^2=\beta^2$, then $y=0$ and $x\in\R\subset Q_0$. If $\alpha^2>\beta^2$, the element $s=(\alpha^2-\beta^2)^{-1/2}y$ belongs to $S_0$ by Lemma \ref{lemma_3}, and $x=\beta+\sqrt{\alpha^2-\beta^2}\,s$ belongs to $Q_0$.
\end{proof}

\subsection{The Euclidean cone and the quadratic cone} The set \[ Q_0=\pi(\mathbb C\times S_0) \] depends on the Euclidean structure fixed on $A$. It should therefore be distinguished from the quadratic cone $Q_A$ associated by Ghiloni--Perotti to a real alternative $*$-algebra. The two cones agree only under additional compatibility assumptions. Let $A$ be endowed with a $*$-involution \[ x\longmapsto x^c. \] The corresponding quadratic cone is defined in terms of the trace and the norm \[ t(x)=x+x^c,\qquad n(x)=xx^c. \] On the other hand, the present construction starts from a Euclidean inner product on $A$ and defines $S_0$ by the orthogonality of $L_s$. There is no reason, in a general real associative algebra, for the adjoint endomorphism $L_x^t$ to be again a left multiplication. In particular, the Euclidean construction does not produce a canonical $*$-involution. A natural compatibility condition is the following: \[ L(A)^t\subseteq L(A). \] If this holds, then for every $x\in A$ there exists a unique element $x^\dagger\in A$ such that \[ L_x^t=L_{x^\dagger}. \] Since \[ L_{xy}^t=L_y^tL_x^t=L_{y^\dagger x^\dagger}, \] the map $x\mapsto x^\dagger$ is an anti-automorphism. If moreover \[ (x^\dagger)^\dagger=x, \] then it is a $*$-involution induced by the Euclidean structure. \begin{propos}\label{prp_metric_involution} Assume that the Euclidean structure on $A$ satisfies \[ L(A)^t\subseteq L(A) \] and let $x\mapsto x^\dagger$ be defined by \[ L_x^t=L_{x^\dagger}. \] Then \[ S_0=\{s\in A:\ s^2=-1,\ s^\dagger=-s\}. \] \end{propos} \begin{proof} If $s\in S_0$, then $s^2=-1$ and $L_s^t=-L_s$ by Lemma \ref{lemma_3}. Since $L_s^t=L_{s^\dagger}$ and the map $a\mapsto L_a$ is injective, we get $s^\dagger=-s$. Conversely, if $s^2=-1$ and $s^\dagger=-s$, then \[ L_s^t=L_{s^\dagger}=-L_s. \] Together with $L_s^2=-I$, this implies \[ L_s^tL_s=(-L_s)L_s=I, \] hence $s\in S_0$. \end{proof} \begin{corol}\label{cor_Q0_QA_compatible} Assume that $A$ is a real associative $*$-algebra and that the Euclidean structure is compatible with the involution in the sense that \[ L_x^t=L_{x^c} \] for all $x\in A$. If \[ S_0=\mathbb S_A, \] where $\mathbb S_A$ is the set of imaginary units of the $*$-algebra, then \[ Q_0=Q_A. \] \end{corol} \begin{proof} Under the compatibility assumption, Proposition \ref{prp_metric_involution} gives \[ S_0=\{s\in A:\ s^2=-1,\ s^c=-s\}. \] If this set coincides with $\mathbb S_A$, then the two cones are obtained by taking the same union of complex slices: \[ Q_0=\bigcup_{s\in S_0}(\mathbb R+\mathbb R s) = \bigcup_{J\in\mathbb S_A}(\mathbb R+\mathbb R J) =Q_A. \] \end{proof} \begin{rem} The equality $S_0=\mathbb S_A$ is not automatic. For instance, in algebras with non-compact sets of imaginary units, the set $\mathbb S_A$ cannot coincide with $S_0$, since $S_0$ is compact by Corollary \ref{cor_S0_compact}. Thus $Q_0$ and $Q_A$ should be regarded as distinct objects in general. \end{rem} The distinction between the two cones reflects two different geometric roles. The quadratic cone $Q_A$ is adapted to the algebraic theory of slice functions on a $*$-algebra. It carries the trace, the norm, the conjugation, and the characteristic spheres which enter the algebraic study of zero sets and the fundamental theorem of algebra for slice-regular polynomials. The Euclidean cone $Q_0$, instead, is adapted to the complex-geometric construction of the present paper. Its parameter space $S_0$ is compact; hence incidence varieties over $S_0$ have proper projections, and Remmert's theorem applies. Thus the role of $Q_0$ is not to replace $Q_A$, but to provide a compact family of complex slices on which the holomorphic and twistor-type machinery is available. \begin{rem} For $x=\alpha+\beta s$ with $s\in S_0$, one may define, relative to the chosen Euclidean cone, the quantities \[ t_0(x)=2\alpha, \qquad n_0(x)=\alpha^2+\beta^2. \] These agree with the usual trace and norm of a $*$-algebra only when the Euclidean structure and the involution are compatible as above. Without such compatibility, $t_0$ and $n_0$ are functions on $Q_0$, not intrinsic trace and norm functions on the whole algebra $A$. \end{rem}

\subsection{Equations for the zero variety}

For a non-compact parameter space, the set-theoretic zero variety defined in Section \ref{zero_var} need not be a complex subvariety. In the compact case, however, Theorem \ref{teo_zero_var} applies. We now focus on the compact subvariety $S_0$ and define
$$Z_0=\{w\in\C\otimes A\ :\ \pi(w,s)=0\textrm{ for some }s\in S_0\}\;,$$
then we can obtain some results. Let $Q:\C\otimes A\times\C\otimes A\to\C$ be the $\C$-bilinear extension of the inner product on $A$ and $\Phi:\C\otimes A\to \C$ be the quadratic form $\Phi(w)=Q(w,w)$.

\begin{propos}\label{prp_cont1}If $w\in \C\otimes A$ belongs to $Z_0$ then $\Phi(ww')=0$ for all $w'\in \C\otimes A$.\end{propos}
\begin{proof}If $w\in Z_0$, then there exists $s\in S_0$ such that $\pi(w,s)=0$. As in the proof of Theorem \ref{teo_zero_var}, we note that $\pi(w,s)=0$ if and only if $w\in \mathcal{W}(L_s)$. Hence
$$w=1\otimes v + i\otimes L_s v$$
for some $v\in A$. Therefore, writing $I_1$ for the identity map in $\mathrm{End}(\C)$ and $I$ for the identity map in $\mathrm{End}(A)$,
$$ww'=(I_1\otimes L_v)w'+i(I_1\otimes L_sL_v)w'\;.$$
We note that
$$((I_1\otimes L_v)+i(I_1\otimes L_sL_v))^t=(I_1\otimes L_v^t)-i(I_1\otimes L_v^tL_s)$$
so
$$((I_1\otimes L_v)+i(I_1\otimes L_sL_v))^t((I_1\otimes L_v)+i(I_1\otimes L_sL_v))=((I_1\otimes L_v^t)-i(I_1\otimes L_v^tL_s))((I_1\otimes L_v)+i(I_1\otimes L_sL_v))=$$
$$=(I_1\otimes L_v)^t(I_1\otimes I +I_1\otimes L_s^2)(I_1\otimes L_v)=0$$
Therefore
$$\Phi(ww')=Q((I_1\otimes L_v)w'+i(I_1\otimes L_sL_v)w',(I_1\otimes L_v)w'+i(I_1\otimes L_sL_v)w')=Q(w',0)=0\;,$$
as we have $Q(T\cdot,\cdot)=Q(\cdot,T^t\cdot)$ for every endomorphism $T$ of $\C\otimes A$.\end{proof}

We note that the set
\begin{equation}\label{eq_luogo}
\mathcal{Z}=\{w\in \C\otimes A\ :\ \Phi(ww')=0\textrm{ for every }w'\in \C\otimes A\}\end{equation}
is given by a finite number of quadratic equations in $\C\otimes A$, namely, by the $N(N+1)/2$ equations $L_w^tL_w=0$, therefore it is a complex subvariety of $\C\otimes A$.

\begin{teorema}\label{teo_luogo}Suppose that the space $\mathcal{Z}$, described by \eqref{eq_luogo}, is irreducible as a complex subvariety of $\C\otimes A$, then it coincides with $Z_0$.
\end{teorema}
\begin{proof}By the previous Proposition, we already know that
$$Z_0\subseteq\{w\in \C\otimes A\ :\ \Phi(ww')=0\textrm{ for every }w'\in \C\otimes A\}=\{w\in\C\otimes A\ :\ L_w^tL_w=0\}\;.$$
On the other hand, if we write $w\in \C\otimes A$  as $w=1\otimes a+i\otimes b$, then
$$L_w^tL_w=(I_1\otimes (L_a^tL_a-L_b^tL_b)) + i(I_1\otimes(L_b^tL_a+L_a^tL_b))\;;$$
so, $L_w^tL_w=0$ if and only if
$$L_a^tL_a=L_b^tL_b\quad\textrm{ and }\quad L_a^tL_b+L_b^tL_a=0\;.$$

If $L_w^tL_w=0$ and $a$ is invertible, then also $b$ is invertible, and vice versa; in such case, we set $u=ba^{-1}$ and it is easy to check that
\begin{enumerate}
\item $L_u^tL_u=I$
\item $L_u^t+L_u=0$
\item $L_b=L_uL_a$.
\end{enumerate}
Therefore, $w\in Z_0$. Now, being invertible is an open condition, so there is an open subset of the space described by \eqref{eq_luogo} which is contained in $Z_0$; by irreducibility, the two complex subvarieties coincide.\end{proof}

\subsection{Explicit equations of \texorpdfstring{$Z_0$}{Z0} in the simple cases}\label{explicit_Z0}

Theorem \ref{teo_luogo} reduces the determination of $Z_0$ to the algebraic variety $\mathcal Z=\{w\in\C\otimes A:L_w^tL_w=0\}$ together with an irreducibility check. For the classical simple algebras, endowed with the Frobenius inner product, both can be carried out explicitly, turning Theorem \ref{teo_luogo} from an abstract criterion into a list of concrete descriptions. Throughout, we use the splitting $w=1\otimes a+i\otimes b$ and the equations
\begin{equation}\label{eq_Z0_split}
L_a^tL_a=L_b^tL_b,\qquad L_a^tL_b+L_b^tL_a=0,
\end{equation}
obtained in the proof of Theorem \ref{teo_luogo}.

\begin{ex}[The quaternions]\label{ex_Z0_H}
Let $A=\H$ with the standard Euclidean structure, so that $L_x^t=L_{\bar x}$, where $x\mapsto\bar x$ is quaternionic conjugation. Then $L_w^tL_w=L_{\bar a a+\bar b b}+\,i\,L_{\bar a b+\bar b a}$ vanishes if and only if
\[
\bar a a=\bar b b,\qquad \bar a b+\bar b a=0,
\]
that is, $|a|^2=|b|^2$ and $a\perp b$ in the real sense $\operatorname{Re}(\bar a b)=0$. Setting $w=1\otimes a+i\otimes b$ and $\Phi(w)=Q(w,w)=|a|^2-|b|^2+2i\,\langle a,b\rangle$, the variety $\mathcal Z$ is the complex quadric cone
\[
\mathcal Z=\{w\in\C\otimes\H:\ \Phi(w)=0\},
\]
a single irreducible quadric in $\C^4$. By Theorem \ref{teo_luogo}, $Z_0=\mathcal Z$ is the isotropic cone of the complexified Euclidean form, recovering the classical fact that in the quaternionic case the zero variety is a hypersurface.
\end{ex}

\begin{ex}[Real matrix algebras]\label{ex_Z0_MR}
Let $A=M_{2m}(\R)$ with the Frobenius inner product, for which $L_X^t=L_{X^t}$. Then \eqref{eq_Z0_split} reads
\[
A^tA=B^tB,\qquad A^tB+B^tA=0,
\]
for $w=1\otimes A+i\otimes B$. Equivalently, the complex matrix $C=A+iB\in M_{2m}(\C)$ satisfies $C^tC=0$, where $C^t$ is the plain transpose. Thus
\[
\mathcal Z=\{C\in M_{2m}(\C):\ C^tC=0\}
\]
is the variety of matrices whose columns span a totally isotropic subspace for the standard symmetric bilinear form on $\C^{2m}$. Its irreducible components are indexed by the rank $r$ of $C$, with $0\le r\le m$, the top component $r=m$ being the closure of the matrices whose column space is a maximal isotropic ($m$-dimensional) subspace. This is a determinantal-type variety fibred over the orthogonal Grassmannian $OG(m,2m)$; on the locus of invertible $A$, equation \eqref{eq_Z0_split} gives $U=BA^{-1}$ with $U^tU=I$ and $U^t=-U$, i.e. $U\in S_0\simeq O(2m)/U(m)$, so $Z_0$ is the irreducible component containing this open set.
\end{ex}

\begin{ex}[Complex matrix algebras as real algebras]\label{ex_Z0_MC}
Let $A=M_n(\C)$ regarded as a real algebra, with the Frobenius Hermitian structure, for which $L_X^t$ corresponds to $L_{X^*}$, $X^*$ the conjugate transpose. Writing $w=1\otimes A+i\otimes B$, equation \eqref{eq_Z0_split} becomes
\[
A^*A=B^*B,\qquad A^*B+B^*A=0.
\]
On the locus where $A$ is invertible, $U=BA^{-1}$ satisfies $U^*U=I$ and $U^*=-U$, hence $U$ is a skew-Hermitian unitary, i.e. $U\in S_0$. The component $S_{0,k}\simeq\operatorname{Gr}(k,n)$ corresponds to the $k$-dimensional $i$-eigenspace of $U$, so $\mathcal Z$ decomposes accordingly and each piece $Z_{0,k}$ is the cone over the tautological bundle of $\operatorname{Gr}(k,n)$. Thus $Z_0=\coprod_k Z_{0,k}$ is a finite union of determinantal cones, one for each Grassmannian component of $S_0$.
\end{ex}

\begin{ex}[Quaternionic matrix algebras]\label{ex_Z0_MH}
Let $A=M_m(\H)$ with the standard Euclidean structure, for which $L_X^t=L_{X^*}$, $X^*=\bar X^t$ the quaternionic conjugate transpose. Then, for $w=1\otimes A+i\otimes B$, \eqref{eq_Z0_split} reads
\[
A^*A=B^*B,\qquad A^*B+B^*A=0,
\]
and on the invertible locus $U=BA^{-1}$ is a quaternionic unitary with $U^*=-U$ and $U^2=-I$, that is $U\in S_0\simeq Sp(m)/U(m)$. The corresponding component of $\mathcal Z$ is the cone over the tautological bundle of $Sp(m)/U(m)$, and Theorem \ref{teo_luogo} identifies $Z_0$ with this irreducible determinantal cone.
\end{ex}

\begin{rem}\label{rem_Z0_explicit}
In each case $Z_0$ is realized as an isotropic or determinantal cone fibred over the corresponding compact Hermitian symmetric space of Theorem \ref{teo_wedderburn}: the quadric cone for $\H$, the totally isotropic matrix variety for $M_{2m}(\R)$, and the determinantal cones over Grassmannians and over $Sp(m)/U(m)$ in the remaining cases. This makes the irreducibility hypothesis of Theorem \ref{teo_luogo} verifiable component by component and yields explicit defining equations \eqref{eq_Z0_split} for the zero variety.
\end{rem}

%

\section{Application: generalized slice-regular functions}\label{application}

As an application of the twistor-geometric machinery of the previous sections, we now drop the requirement that the holomorphic lift of a slice-regular function be of the special form $\mathscr F(z,s)=(F(z),s)$, and study the resulting larger class of functions.

Consider the following definition.

\begin{defin}Let $A$ be a real associative algebra, $\cU\subseteq \C$ an open domain and $V=\pi(\cU\times S)$; a function $f:V\to A$ is called a \emph{generalized slice-regular function} if there exists a holomorphic map $\mathfrak{F}:\cU\times S\to\C\otimes A\times S$ such that $f\circ\pi=\pi\circ \mathfrak{F}$ and $\mathfrak{F}\circ\sigma=\tau\circ\mathfrak{F}$.\end{defin}

\begin{lemma}\label{lmm_generalized_stem}Let $A$ be an associative algebra and suppose that $S$ is compact and connected. If $f:V\to A$ is a generalized slice-regular function, there exist a holomorphic function $F:\cU\to\C\otimes A$ and a holomorphic map $\Phi:\cU\times S\to S$ such that
$$f(\pi(z,s))=\pi(F(z), \Phi(z,s))\;.$$
\end{lemma}
\begin{proof}We write $\mathfrak{F}=(F,\Phi)$ with $F:\cU\times S \to\C\otimes A$ and $\Phi:\cU\times S\to S$ holomorphic. For each fixed $z\in\cU$, the map $s\mapsto F(z,s)$ is a holomorphic map from the compact connected complex manifold $S$ to $\C\otimes A$. Since $\C\otimes A$ is a finite-dimensional complex vector space, hence a Stein manifold, every holomorphic map from a compact connected complex manifold to it is constant (by the maximum principle applied to the coordinate functions). Therefore $F(z,s)=F(z)$ does not depend on $s$, and we may regard $F$ as a holomorphic function $F:\cU\to\C\otimes A$.
\end{proof}

The classical definition of slice-regular function is obtained with $\Phi(z,s)=s$ for all $(z,s)\in\cU\times S$. We first record the basic structural form of a generalized slice-regular function, which follows at once from Lemma \ref{lmm_generalized_stem} together with the equivariance condition $\mathfrak F\circ\sigma=\tau\circ\mathfrak F$.

\begin{propos}\label{prp_gen_structure}
Let $A$ be associative with $S$ compact and connected, and let $f:V\to A$ be a generalized slice-regular function. Write $F=\alpha+i\beta$ with $\alpha,\beta:\cU\to A$. Then
\begin{equation}\label{eq_gen_form}
f(\pi(z,s))=\alpha(z)+\Phi(z,s)\,\beta(z),
\end{equation}
where $\alpha,\beta$ satisfy the conditions of Section \ref{sregA}, namely $\alpha(\bar z)=\alpha(z)$, $\beta(\bar z)=-\beta(z)$ and the Cauchy--Riemann system $\partial_x\alpha=\partial_y\beta$, $\partial_y\alpha=-\partial_x\beta$. Moreover the angular component satisfies
\begin{equation}\label{eq_gen_Phi}
\Phi(\bar z,-s)=-\Phi(z,s)\qquad\text{for all }(z,s)\in\cU\times S.
\end{equation}
\end{propos}
\begin{proof}
By Lemma \ref{lmm_generalized_stem} we have $f(\pi(z,s))=\pi(F(z),\Phi(z,s))$ with $F:\cU\to\C\otimes A$ holomorphic. Writing $F=\alpha+i\beta$ and using $\pi(\alpha(z)+i\beta(z),w)=\alpha(z)+w\beta(z)$ for $w\in S$ gives \eqref{eq_gen_form}. Holomorphicity of $F$ is the Cauchy--Riemann system, and the identity $\mathfrak F\circ\sigma=\tau\circ\mathfrak F$ reads, on the first component, $F(\bar z)=\overline{F(z)}$, i.e. $\alpha(\bar z)=\alpha(z)$ and $\beta(\bar z)=-\beta(z)$, and, on the second component, exactly \eqref{eq_gen_Phi}.
\end{proof}

If one is interested in more quantitative properties, like representation formulas or a differential characterization, the form of the function $\Phi$ comes into play; we can put some restrictions on the form of $\Phi$ and obtain some subclasses of generalized slice-regular functions.

\paragraph{1.} If we ask that for every $z\in\cU$ the map $s\mapsto \Phi(z,s)$ is injective (or even an automorphism of $S$), we obtain functions which have generically isolated zeros.

\paragraph{2.} If we fix a holomorphic map $z\mapsto \phi_z\in\mathrm{Aut}(S)$, we can consider all the functions $f$ for which $\Phi(z,s)=\phi_z(s)$.

\paragraph{3.} If we fix $\phi\in\mathrm{Aut}(S)$, we can consider all the functions $f$ for which $\Phi(z,s)=\phi(s)$.

\medskip

The third class is the most rigid one, and for it the heuristic properties listed above can be turned into theorems. The key observation is that, for a fixed automorphism $\phi$, a function of the third class is obtained from an ordinary slice-regular function by a biholomorphic reparametrization of the slice variable. We make this precise.

\subsection{The reparametrization principle for the third class}\label{repar_class}

Let $\phi\in\mathrm{Aut}(S)$ be a holomorphic automorphism of $(S,J_S)$. Recall from Remark \ref{rem_iperpiano} that the map $(x+iy,s)\mapsto x+sy$ is injective on $\C_+\times S$; consequently every $a\in V\setminus\R$ is written uniquely as $a=\pi(z,s)$ with $z\in\cU\cap\C_+$ and $s\in S$, while on $V\cap\R$ the value $\pi(x,s)=x$ does not depend on $s$. We may therefore define a map
\[
\Psi_\phi:V\longrightarrow V,\qquad \Psi_\phi(\pi(z,s))=\pi(z,\phi(s)),\quad \Psi_\phi(x)=x\ \ (x\in V\cap\R).
\]

\begin{lemma}\label{lmm_reparam}
The map $\Psi_\phi$ is a well-defined homeomorphism of $V$, it restricts to a biholomorphism of $(V\setminus\R,\J)$, where $\J$ is the complex structure of Appendix \ref{hol_quat} extended to the present setting, and it maps the slice $\pi(\cU\times\{s\})$ onto the slice $\pi(\cU\times\{\phi(s)\})$. Its inverse is $\Psi_{\phi^{-1}}$.
\end{lemma}
\begin{proof}
Well-definedness and bijectivity on $V\setminus\R$ follow from the uniqueness of the representation $a=\pi(z,s)$, $z\in\C_+$, together with the fact that $\phi$ is a bijection of $S$; on $V\cap\R$ the map is the identity. By construction $\Psi_\phi$ sends $\pi(\cU\times\{s\})$ to $\pi(\cU\times\{\phi(s)\})$ and $\Psi_\phi\circ\Psi_{\phi^{-1}}=\mathrm{id}$. Finally, under the biholomorphism $(V\setminus\R,\J)\cong(\cU\cap\C_+)\times S$ recalled in Appendix \ref{hol_quat}, the map $\Psi_\phi$ corresponds to $\mathrm{id}\times\phi$; since $\phi$ is a biholomorphism of $(S,J_S)$, so is $\mathrm{id}\times\phi$, hence $\Psi_\phi$ is a biholomorphism of $V\setminus\R$.
\end{proof}

\begin{propos}\label{prp_repar}
Let $f:V\to A$ belong to the third class, $\Phi(z,s)=\phi(s)$ with $\phi\in\mathrm{Aut}(S)$, and let $g=\mathcal I(F):V\to A$ be the ordinary slice-regular function induced by the same stem function $F$. Then
\begin{equation}\label{eq_repar}
f=g\circ\Psi_\phi .
\end{equation}
Conversely, for every ordinary slice-regular $g$ and every $\phi\in\mathrm{Aut}(S)$, the function $g\circ\Psi_\phi$ is a generalized slice-regular function of the third class. Thus the third class is precisely $\{\,g\circ\Psi_\phi:\ g\ \text{slice-regular},\ \phi\in\mathrm{Aut}(S)\,\}$.
\end{propos}
\begin{proof}
By \eqref{eq_gen_form} we have $f(\pi(z,s))=\alpha(z)+\phi(s)\beta(z)=\pi(F(z),\phi(s))=g(\pi(z,\phi(s)))=g(\Psi_\phi(\pi(z,s)))$, which is \eqref{eq_repar}. Conversely, given $g=\mathcal I(F)$ and $\phi\in\mathrm{Aut}(S)$, set $\mathfrak F(z,s)=(F(z),\phi(s))$; this is holomorphic on $\cU\times S$ and satisfies $\pi\circ\mathfrak F=(g\circ\Psi_\phi)\circ\pi$. The equivariance $\mathfrak F\circ\sigma=\tau\circ\mathfrak F$ holds because $F(\bar z)=\overline{F(z)}$ and $\phi(-s)=-\phi(s)$, the oddness of $\phi$ being established in full generality in Remark \ref{rem_antipodal} below. Hence $g\circ\Psi_\phi$ is generalized slice-regular of the third class.
\end{proof}

\begin{rem}\label{rem_antipodal}
The oddness $\phi(-s)=-\phi(s)$ used in the proof of Proposition \ref{prp_repar} deserves a precise justification, since the antipodal map
\[
\nu:S\longrightarrow S,\qquad \nu(s)=-s,
\]
is \emph{not} itself an element of $\mathrm{Aut}(S,J_S)$. Indeed $\nu$ is a fixed-point-free real-analytic involution of $S$ (each component being mapped to a component of the same dimension, as $\tr K_{-s}=\tr K_s$), and its differential is $d\nu_s=-\mathrm{Id}$ on $A$. Since $T_{-s}S=\{h:(-s)h+h(-s)=0\}=T_sS$, the map $d\nu_s$ does send $T_sS$ to $T_{-s}S$; however, comparing it with the complex structures \eqref{eq_cpstr} of Appendix \ref{hol_quat}, namely $J_S^{(s)}h=sh$ and $J_S^{(-s)}h=-sh$, we find
\[
d\nu_s\big(J_S^{(s)}h\big)=-(sh),\qquad J_S^{(-s)}\big(d\nu_s h\big)=(-s)(-h)=sh,
\]
so that $d\nu_s\circ J_S^{(s)}=-\,J_S^{(-s)}\circ d\nu_s$. Thus $\nu$ is \emph{anti}-holomorphic for $J_S$. Consequently the oddness of $\phi$ cannot follow from a commutation of $\nu$ with the holomorphic automorphism group as a whole.

What is true, and what the argument actually requires, is that $\nu$ commutes with the automorphisms relevant to the construction, all of which are restrictions to $S$ of $\R$-linear self-maps of $A$. The reparametrizations $\phi$ admissible in the third class are precisely those compatible with the algebra structure underlying $\pi(z,s)=x+sy$, hence arise as restrictions to $S$ of $\R$-linear automorphisms of $A$ preserving $S$. These are of two kinds. First, the inner automorphisms
\[
\phi_g(s)=g^{-1}sg,\qquad g\in A^\times,
\]
which exhaust a neighbourhood of the identity in the local picture of $S$ as a conjugacy orbit (Lemma \ref{lemma_S_smooth}) and act holomorphically on $(S,J_S)$, being induced by left/right multiplications that commute with $L_s$; second, when $n=p+q$ is odd and $p-q\equiv 1\bmod 4$, the central twist $s\mapsto\omega s$ with $\omega=e_1\cdots e_n$, $\omega^2=1$ central, considered in the proof of Corollary \ref{cor_Cliff2}. In either case the map is $\R$-linear, and since $\nu=-\mathrm{Id}$ is $\R$-linear it commutes with it:
\[
\phi_g(\nu(s))=g^{-1}(-s)g=-\,g^{-1}sg=\nu(\phi_g(s)),\qquad \omega(\nu(s))=\omega(-s)=-(\omega s)=\nu(\omega s).
\]
Therefore every admissible $\phi$ satisfies $\phi(-s)=-\phi(s)$, as claimed.

Finally, this argument is manifestly \emph{invariant across the Clifford signatures} $(p,q)$: oddness rests only on the $\R$-linearity of $\nu$ and of the maps $\phi$, never on the dimension, the parity of $n$, or the residue of $p-q$. The signature enters the geometry of $S$ elsewhere—through $\tr K_s$ and the dimensions of the components (Corollaries \ref{cor_Cliff1}, \ref{cor_Cliff2} and the subsequent Remarks)—but it never affects the equivariance $\phi\circ\nu=\nu\circ\phi$. Hence Lemma \ref{lmm_reparam} and Proposition \ref{prp_repar} hold without restriction on $(p,q)$, in every case in which generalized slice-regular functions are defined.
\end{rem}

Proposition \ref{prp_repar} is an effective tool: every property of ordinary slice-regular functions that is invariant under the slice-preserving biholomorphism $\Psi_\phi$ transfers verbatim to the third class. We now harvest several such properties.

\begin{teorema}[Maximum modulus principle]\label{teo_max_modulus}
Let $\|\cdot\|$ be any norm on $A$ and let $f:V\to A$ be a generalized slice-regular function of the third class, with $\cU$ bounded. Then $\|f\|$ attains no strict local maximum in the interior of $V$ unless $f$ is constant; moreover
\[
\sup_{a\in V}\|f(a)\|=\sup_{a\in\partial_S V}\|f(a)\|,\qquad \partial_S V:=\pi(\partial\cU\times S).
\]
\end{teorema}
\begin{proof}
Fix $s\in S$ and consider the restriction of $f$ to the slice $\pi(\cU\times\{s\})$. By Proposition \ref{prp_repar}, $f(\pi(z,s))=\pi(F(z),\phi(s))$. Identifying the slice $\C_{\phi(s)}$ with $\C$ through $x+\phi(s)y\leftrightarrow x+iy$ and endowing $A$ with the complex structure $L_{\phi(s)}$, the map $z\mapsto\pi(F(z),\phi(s))$ is holomorphic, since $F$ is holomorphic and $\pi(\cdot,\phi(s))$ is complex linear from $(\C\otimes A,J_0\otimes\Uno)$ to $(A,L_{\phi(s)})$. A holomorphic map with values in the finite-dimensional complex vector space $(A,L_{\phi(s)})$ has plurisubharmonic norm, so $\|f\|$ restricted to each slice is subharmonic and obeys the maximum principle on $\cU$. Since every point of $V$ lies on such a slice and the slice boundary is contained in $\partial_S V$, the global statement follows by taking the supremum over $s\in S$.
\end{proof}

\begin{teorema}[Representation formula]\label{teo_representation}
Let $f:V\to A$ belong to the third class, with $\Phi(z,s)=\phi(s)$. Let $a,b\in S$ be such that $\phi(a)-\phi(b)\in A^\times$. Then, for every $s\in S$ and every $z\in\cU$,
\begin{equation}\label{eq_representation}
f(\pi(z,s))=f(\pi(z,a))+\bigl(\phi(s)-\phi(a)\bigr)\bigl(\phi(a)-\phi(b)\bigr)^{-1}\bigl(f(\pi(z,a))-f(\pi(z,b))\bigr).
\end{equation}
In particular $f$ is determined on all of $V$ by its values on the two slices $\pi(\cU\times\{a\})$ and $\pi(\cU\times\{b\})$.
\end{teorema}
\begin{proof}
Write $f_a:=f(\pi(z,a))=\alpha(z)+\phi(a)\beta(z)$ and $f_b:=f(\pi(z,b))=\alpha(z)+\phi(b)\beta(z)$, by \eqref{eq_gen_form}. Then $f_a-f_b=(\phi(a)-\phi(b))\beta(z)$, and since $\phi(a)-\phi(b)$ is invertible we obtain $\beta(z)=(\phi(a)-\phi(b))^{-1}(f_a-f_b)$ and $\alpha(z)=f_a-\phi(a)\beta(z)$. Substituting into $f(\pi(z,s))=\alpha(z)+\phi(s)\beta(z)$ gives \eqref{eq_representation}.
\end{proof}

\begin{rem}\label{rem_cauchy_slice}
The holomorphy of $z\mapsto\pi(F(z),\phi(s))$ established in the proof of Theorem \ref{teo_max_modulus} yields, on every slice, an honest Cauchy integral formula: if $\overline{D}\subset\cU$ is a closed disc with $z$ in its interior, then
\[
f(\pi(z,s))=\frac{1}{2\pi i}\oint_{\partial D}\frac{\pi(F(\zeta),\phi(s))}{\zeta-z}\,d\zeta,
\]
the integral being taken in the complex line $\C_{\phi(s)}\subset A$. This is the representation result anticipated above for the third class.
\end{rem}

\begin{corol}[Identity principle and structure of the zero set]\label{cor_identity}
Let $f$ be of the third class on $V=\pi(\cU\times S)$ with $\cU$ a domain. If $f$ vanishes on a subset of some slice $\pi(\cU\times\{s_0\})$ having an accumulation point in $\cU$, then $f\equiv 0$. More generally, the zero set of $f$ is the image under $\Psi_{\phi}^{-1}$ of the zero set of the ordinary slice-regular function $g=\mathcal I(F)$; in particular, on each slice the zeros of $f$ are isolated unless $f$ vanishes identically on that slice, and the non-isolated zeros of $f$ form a discrete union of \emph{characteristic sets} $\pi(\{z_0\}\times S)$.
\end{corol}
\begin{proof}
By Proposition \ref{prp_repar}, $f=g\circ\Psi_\phi$ with $\Psi_\phi$ a slice-preserving biholomorphism of $V\setminus\R$ (Lemma \ref{lmm_reparam}); hence $f^{-1}(0)=\Psi_\phi^{-1}(g^{-1}(0))$. On the slice $\pi(\cU\times\{s_0\})$ the function $z\mapsto\pi(F(z),\phi(s_0))$ is holomorphic with values in $(A,L_{\phi(s_0)})$, so the classical identity principle for vector-valued holomorphic functions applies and the stated conclusions follow from the corresponding facts for $g$, established in \cites{GP1,M1}.
\end{proof}

\begin{rem}\label{rem_rouche_H}
When $A=\H$, the function $g=\mathcal I(F)$ is an ordinary quaternionic slice-regular function, for which a Rouché theorem, an argument principle and a Hurwitz theorem are available \cite{GSS1}. Through the biholomorphism $\Psi_\phi$ of Proposition \ref{prp_repar} these statements transfer to every $f=g\circ\Psi_\phi$ of the third class, the relevant winding numbers being computed on the reparametrized slices $\pi(\cU\times\{\phi(s)\})$.
\end{rem}

\subsection{A differential characterization}\label{pde_char}

The definition of generalized slice regularity is, by its very nature, a holomorphy condition on the lift $\mathfrak F=(F,\Phi)$. We now express it, in the spirit of the global operators of Colombo--Sabadini--Struppa \cite{CSS} and Ghiloni--Perotti \cite{GhP}, as the vanishing of a Cauchy--Riemann type operator. We work on $V\setminus\R$, which carries the complex structure $\J$ of Appendix \ref{hol_quat}: for $a=\pi(z,s)\in V\setminus\R$ and $x\in T_a(V\setminus\R)$, one has $\J_a(x)=L_s(x)=sx$.

The general statement reads as follows.

\begin{propos}\label{prp_pde_general}
Let $A$ be associative with $S$ compact and connected. A function $f:V\to A$ is generalized slice-regular if and only if its lift $\widehat f:=\mathfrak F\circ(\pi|_{\C_+\times S})^{-1}:V\setminus\R\to\C\otimes A\times S$ is holomorphic from $(V\setminus\R,\J)$ to $(\C\otimes A\times S,\,(J_0\otimes\Uno)\oplus J_S)$, that is
\[
\bar\partial_{\J}\,\widehat f=0\qquad\text{on }V\setminus\R .
\]
\end{propos}
\begin{proof}
Under the biholomorphism $(V\setminus\R,\J)\cong(\cU\cap\C_+)\times S$ recalled in Appendix \ref{hol_quat}, the lift $\widehat f$ corresponds to $\mathfrak F=(F,\Phi)$. Holomorphy of $\mathfrak F$ for the product complex structure $(J_0\otimes\Uno)\oplus J_S$ is precisely the holomorphy required in the definition of generalized slice regularity; the involutive constraints $\mathfrak F\circ\sigma=\tau\circ\mathfrak F$ extend the data across $\R$ and are encoded in the equivariance of $\widehat f$. Hence $\bar\partial_{\J}\widehat f=0$ is equivalent to $f$ being generalized slice-regular.
\end{proof}

For the third class the condition descends to a single first-order operator acting on $f$ itself, with no reference to the lift.

\begin{teorema}[Twisted Cauchy--Riemann operator]\label{teo_pde_class3}
Fix $\phi\in\mathrm{Aut}(S)$ and let $f:V\to A$ be a slice function of the form \eqref{eq_gen_form} with $\Phi(z,s)=\phi(s)$, of class $C^1$. For $a=\pi(z,s)\in V\setminus\R$, with $z=x+iy$, define
\begin{equation}\label{eq_twisted_CR}
\mathcal D_\phi f(a):=\frac{\partial f}{\partial x}(a)+\phi(s)\,\frac{\partial f}{\partial y}(a),
\end{equation}
where $\partial_x,\partial_y$ denote the derivatives along the slice $\C_s$ in the coordinates $z=x+iy$. Then $f$ is a generalized slice-regular function of the third class with angular part $\phi$ if and only if
\[
\mathcal D_\phi f=0\qquad\text{on }V\setminus\R .
\]
For $\phi=\mathrm{id}_S$ the operator $\mathcal D_\phi$ reduces to the slice Cauchy--Riemann operator $\partial_x+L_s\partial_y$, and one recovers the classical definition of slice regularity given in Section \ref{sregA}.
\end{teorema}
\begin{proof}
On the slice $\C_s$ we have $f(\pi(z,s))=\alpha(z)+\phi(s)\beta(z)$, whence
\[
\mathcal D_\phi f=(\partial_x\alpha+\phi(s)\partial_x\beta)+\phi(s)(\partial_y\alpha+\phi(s)\partial_y\beta)=(\partial_x\alpha-\partial_y\beta)+\phi(s)(\partial_x\beta+\partial_y\alpha),
\]
using $\phi(s)^2=-1$. This expression vanishes for all $s$ if and only if $\partial_x\alpha=\partial_y\beta$ and $\partial_y\alpha=-\partial_x\beta$, that is, if and only if $F=\alpha+i\beta$ is holomorphic, which by Proposition \ref{prp_gen_structure} characterizes generalized slice regularity of the third class with $\Phi=\phi$. The reduction for $\phi=\mathrm{id}_S$ is immediate from \eqref{eq_twisted_CR}.
\end{proof}

\begin{rem}\label{rem_global_operator}
As in the classical theory, the slice operator \eqref{eq_twisted_CR} can be rewritten as a global operator with non-constant coefficients on $V\setminus\R$. Indeed, for $a=\pi(z,s)$ the slice unit is recovered intrinsically from the splitting $a=\tr(a)+s\,|a-\tr(a)|$ of Section \ref{orto}, so that $\phi(s)$, and hence $\mathcal D_\phi$, becomes a differential operator whose coefficients depend real-analytically on $a$ on $V\setminus\R$. For $\phi=\mathrm{id}$ this is the slice restriction of the global operator $\vartheta$ of \cite{GhP}, and, in the slice-monogenic setting, of the operator $G$ of \cite{CSS}.
\end{rem}

\medskip

Suppose now that we have two associative algebras $A_1$ and $A_2$, with $S_1$, $S_2$ associated sets of square roots of $-1$; we define the functions $\pi_1:\C\times S_1\to A_1$ and $\pi_2:\C\otimes A_2\times S_2\to A_2$ as above. Consider also the involutions $\sigma_1:\C\times S_1\to \C\times S_1$ and $\tau_2:\C\otimes A_2\times S_2\to \C\otimes A_2\times S_2$.

\begin{defin}Let $\cU\subseteq \C$ an open domain and $V_1=\pi_1(\cU\times S_1)$; a function $f:V_1\to A_2$ is called a \emph{generalized slice-regular function} if there exists a holomorphic map $\mathfrak{F}:\cU\times S_1\to\C\otimes A_2\times S_2$ such that $f\circ\pi_1=\pi_2\circ \mathfrak{F}$ and $\mathfrak{F}\circ\sigma_1=\tau_2\circ\mathfrak{F}$.\end{defin}

As before, such a function is of the form $\mathfrak{F}(z,s)=(F(z), \Phi(z,s))$ and we can define analogues of the three classes above.

\appendix
\section{Holomorphicity in the quaternionic case}\label{hol_quat}

This appendix collects, for the reader's convenience and to keep the main text centered on the twistor space $S(A)$ of a general algebra, the classical quaternionic picture from which the constructions of the paper originated.

This section is intended as a brief introduction to the different ways of defining slice-regular functions (on the quaternions or on a real alternative algebra with an involution) and of their links to actual holomorphic maps; we want to emphasize the connections between the various definitions and to work out the explicit correspondence between the holomorphic stem function and the twistor transform, both associated to a slice-regular function on quaternions.

\medskip

Let $\H$ be the algebra of quaternions and let $\sfera$ be the $2$-sphere of square roots of $-1$, i.e.
$$\sfera=\{q\in\H\ :\ q^2=-1\}\;.$$
Consider the map $\pi:\C\times\sfera\to\H$ given by $\pi((x+\iota y),u)=x+uy$. On $\C\times\sfera$, the map $\sigma:\C\times\sfera\to\C\times\sfera$ given by $\sigma(z,u)=(\overline{z},-u)$ is an involution with no fixed points and $\pi\circ\sigma=\pi$; let $\cU\subseteq\C$ be an open set, symmetric with respect to the real axis (i.e. invariant under complex conjugation), then $\cU\times\sfera$ is invariant under $\sigma$.

The open sets $V\subseteq\H$ that can be obtained as $V=\pi(\cU\times\sfera)$, with $\cU$ as above, are called \emph{axially symmetric}. 

Given $V\subseteq\H$, a function $f:V\to \H$ is said to be a \emph{left slice function} if we can find $\alpha,\beta:\C\to\H$ such that
\begin{itemize}
\item $\alpha(\overline{z})=\alpha(z)$
\item $\beta(\overline{z})=-\beta(z)$
\item $f\circ\pi(z, u)=\alpha(z)+u\beta(z)$ for all $z\in\cU$, $u\in\sfera$.
\end{itemize}
We note that the first two conditions are needed to ensure that $f$ is well defined.

Let $\C\otimes\H$ be the tensor product of $\C$ and $\H$ over $\R$; we define an involution on $\C\otimes\H$ by requiring that $\overline{z\otimes q}=\overline{z}\otimes q$ and taking the $\R$-linear extension. Let $\tau:\C\otimes\H\times\sfera\to\C\otimes\H\times\sfera$ the involution given by $\tau(z\otimes q, u)=(\overline{z}\otimes q, -u)$ and linear on the first component. The space $\C\otimes\H$ has a natural complex structure given by $J_0\otimes\Uno$, where $J_0$ is the natural complex structure of $\C$ induced by the multiplication by $i$ and $\Uno$ is the identity on $\H$.

Moreover, we extend the map $\pi$ to $\C\otimes\H\times\sfera$ as follows $\pi(z\otimes q, u)=\pi(z,u)q$, where the product on the left hand side is in $\H$, and then extend it linearly on $\C\otimes\H\times\sfera$.
We denote this extension again by $\pi$: the latter $\pi$ is an extension of the former once we identify $\C\times\sfera$ with $\C\otimes 1\times\sfera\subset\C\otimes\H\times\sfera$.

In \cite{GP1}, Ghiloni and Perotti prove the following.

\begin{propos}\label{prp_GP}Let $V\subseteq\H$ an axially symmetric domain, $V=\pi(\cU\times\sfera)$, and let $f:V\to\H$ be a left slice function, then there exists a unique $F:\cU\to\C\otimes \H$ such that
\begin{itemize}
\item $F(\overline{z})=\overline{F(z)}$ for all $z\in\cU$,
\item $f\circ\pi(z,u)=\pi(F(z), u)$ for all $(z,u)\in\cU\times\sfera$.
\end{itemize}
We say that $f$ is induced by $F$ and we write $f=\mathcal{I}(F)$; every left slice function can be obtained this way and every function obtained as $\mathcal{I}(F)$ is left slice. Moreover, $f$ is left slice-regular if and only if $F$ is holomorphic from $(\cU, J_0)$ to $(\C\otimes\H, J_0\otimes \Uno)$.
\end{propos}

\begin{rem}The first condition can be restated as an equivariance with respect to the involutions defined above as follows
$$(F,\mathrm{Id}_{\sfera})\circ \sigma=\tau\circ (F,\mathrm{Id}_\sfera)$$
where $(F,\mathrm{Id}_{\sfera}):\cU\times\sfera\to\C\otimes\H\times\sfera$ is given by $(F,\mathrm{Id}_\sfera)(z,u)=(F(z),u)$.
\end{rem}
In other words, they prove that there exists a holomorphic function $F:\cU\to\C\otimes\H$ such that the diagram
$$\xymatrix{\cU\ar[r]^-F\ar[d]_{\pi(\cdot, u)}&\C\otimes \H\ar[d]^{\pi(\cdot, u)}\\V\ar[r]^f&\H}$$
commutes for every $u\in\sfera$.

\medskip

In \cite{M1}, it was shown that, for every $q\in\H$, there exists a complex hypersurface $V(q)\subset\C\otimes\H$ such that $F(z)\in V(q)$ if and only if there exists $u\in\sfera$ such that $f(\pi(z,u))=q$. As $q$ varies in $\H$, the hypersurfaces $V(q)$ change by a translation of a real vector.

Also in \cite{M1}, we constructed a diffeomorphism between $\sfera$ and a complex submanifold of the Grassmannian of $2$-planes in $\C^4$, interpreting $u\in\sfera$ as a linear complex structure on $\R^4$ and associating it to its $(-i)$-eigenspace. This gives a complex structure on $\sfera$ (which can only be the standard one on $\CP^1$), naturally associated to slice-regular functions; we define the (almost) complex structure $\J$ on $\H\setminus\R$ as follows: given $q\in \H\setminus\R$, there exist unique $z\in\C_+$, $u\in\sfera$ such that $\pi(z,u)=q$, then we consider on $T_q(\H\setminus\R)\cong\H$ the linear involution $L_u(x)=ux$, with $x\in\H$, so, we set
\begin{equation}\label{eq_cpstr}\J_q(x)=L_u(x)\qquad\forall\;q\in\H\setminus\R, \ \forall\;x\in T_q(\H\setminus\R)\cong\H\;.\end{equation}
It is easy to show that $\J$ is integrable and that $(\H\setminus\R, \J)$ is biholomorphic to $\C_+\times\sfera$, where $\sfera$ has the standard complex structure of $\CP^1$. These considerations are carried out in detail in \cite{GSS}.

The following result is an easy consequence of the remarks we just made.

\begin{lemma}\label{lmm_hol2}Let $V=\pi(\cU\times\sfera)\subseteq\H$ an axially symmetric domain, $f:V\to\H$ a slice-regular function.  Then there exists a unique function
$$\mathscr{F}:\cU\times\sfera\to\C\otimes\H\times\sfera$$
such that
\begin{itemize}
\item $\mathscr{F}\circ\sigma=\tau\circ\mathscr{F}$
\item $\mathscr{F}$ is holomorphic, where $\sfera$ has been given the standard structure of $\CP^1$
\item $\mathscr{F}$ is the identity on the second component
\item $f\circ\pi=\pi\circ \mathscr{F}$.
\end{itemize}
In particular, $\mathscr{F}$ descends to a holomorphic function from $(V\setminus\R, \J)$ to $\C\otimes\H\times\CP^1$. Conversely, every such function $\mathscr{F}$ induces a slice-regular function $f$.
\end{lemma}

The situation is encoded in the following commutative diagram.
\begin{equation}\label{diagramma1}
\xymatrix{\cU\times\sfera\ar[r]^-{\mathscr{F}}\ar[d]_{\pi} & {\C\otimes\H\times\sfera}\ar[d]^\pi\\ V\ar[r]^f & \H}
\end{equation}
The previous Lemma is somehow halfway between stem functions and the twistor transform. From this viewpoint, slice-regular functions on an axially symmetric open domain $V=\pi(\cU\times\sfera)$ become holomorphic functions from $\cU\times\CP^{1}$ to $\C\otimes\H\times\CP^{1}$; however, as soon as we look at the values of a slice-regular function, we encounter the obstruction given by the projection map $\pi:\C\otimes\H\times\CP^{1}\to\H$, which is not holomorphic in any natural complex structure on the target.

\subsection{Fixing a basis}\label{fixing_basis}

We now look at another way of associating a holomorphic function to a slice-regular function. Let $\{1,I,J,K\}$ be an orthonormal basis for $\H$ such that $IJ=K$; in fact, we will just need that $IJ+JI=0$ and that $IJ=K$, the orthonormality is a consequence of these two properties. Given $z\in\C$, we denote $\pi(z,u)$ by $z_u$ for any $u\in\sfera$; accordingly, we write $\C_u$ for the set of all the quaternions $z_u$, as $z$ varies in $\C$. Moreover, we have that
$$\H\cong \C_I\oplus(\C_IJ)\;,$$
i.e. we have an isomorphism of additive groups $\gamma:\C^2\to\H$ given by $\gamma(z,w)=z_I+w_IJ$.

\begin{rem}\label{rem_ovvio}It is obvious that $\gamma(\lambda z,\lambda w)=\lambda_I\gamma(z,w)$ for all $\lambda\in\C$.\end{rem}

Now, let us consider two elements $u,v\in\sfera$ and $z\in\C$; the quaternions $z_u$ and $z_v$ differ by a rotation of $\H$ that fixes the real axis, i.e. by a transformation of the form
$$Q_p(q)=p^{-1}qp$$
for some $p\in\H\setminus\{0\}$. If we write $p=\gamma(z,w)$, $p'=\pi(z',w')$, given a quaternion $q=z_I$ it is quite clear that $Q_p(q)=Q_{p'}(q)$ if and only if $(z,w)=\lambda(z',w')$ with $\lambda\in\C^*$; therefore the set of maps $Q_p\vert_{\C_I}$ is identified with $\CP^1$, via $\gamma$. 

We define the map
$$\rho:\C\times\CP^1\to\H$$
as follows
$$\rho(z,[u_0:u_1])=Q_{\gamma(u_0,u_1)}(z_I))=(\gamma(u_0,u_1))^{-1}z_I(\gamma(u_0,u_1))\;.$$
By Remark \ref{rem_ovvio},
$$\rho(z,[u_0:u_1])=(\gamma(u_0,u_1))^{-1}\gamma(zu_0,zu_1)$$
is actually a function of the homogeneous $4$-tuple $[u_0:u_1:zu_0:zu_1]$. Let
$$S_1:\CP^1\times\CP^1\to\CP^3$$
be the Segre embedding
$$S_1([x:y],[a:b])=[xa:xb:ya:yb]$$
and let $\rho_1:\CP^3\to\H$ be given by
$$\rho_1([w_0:w_1:w_2:w_3])=\gamma(w_0,w_1)^{-1}\gamma(w_2,w_3)\;.$$
Then $\rho(z,[u_0:u_1])=\rho_1\circ S_1([1,z],[u_0:u_1])$. We have thus factorized the leftmost vertical arrow in \eqref{diagramma1} as
$$\xymatrix{\C\times \sfera\ar[r]^-\cong&\C\times\CP^1\ar[r]^-{S_1}&\CP^3\ar[r]^-{\rho_1}&\H}\;.$$

\medskip

On the other hand, considering the map $\pi:\C\otimes\H\times\sfera\to\H$, we obtain a map
$$\delta: \C\otimes\H\times\CP^1\to\H$$
which, on the elements $(z\otimes q,[u_0:u_1])$ acts as follows
$$\delta(z\otimes q,[u_0:u_1])=\rho(z,[u_0:u_1])q=\gamma(u_0,u_1)^{-1}z_I\gamma(u_0,u_1)q\;.$$

\begin{rem}\label{rem_ovvio2} For every $(z,w)\in\C^2$, we have
$$\gamma(z,w)I=\gamma(iz,-iw)\qquad \gamma(z,w)J=\gamma(-w,z)\qquad \gamma(z,w)K=\gamma(iw,iz)\;.$$
\end{rem}
Thanks to Remarks \ref{rem_ovvio}, \ref{rem_ovvio2} and to the additivity of $\gamma$, we obtain that $\delta(z\otimes q, [u_0:u_1])$ is
$$\gamma(u_0,u_1)^{-1}\gamma(q_0zu_0+iq_1zu_0-q_2zu_1+iq_3zu_1,q_0zu_1-iq_1zu_1+q_2zu_0+iq_3zu_0)$$
where $q=q_0+q_1I+q_2J+q_3K$. If we use the basis $\{1,I,J,K\}$ to identify $\C\otimes\H$ with $\C^4$, then $\delta((z_0,z_1,z_2,z_3),[u_0:u_1])=\rho_1\circ \delta_1\circ S_2([1:z_0:z_1:z_2:z_3],[u_0:u_1])$, where
$$S_2:\CP^4\times\CP^1\to\CP^9$$
is the Segre embedding and
$$\delta_1:\CP^9\to\CP^3$$
is given by
$$\delta_1([w_0:\ldots:w_9])=[w_0:w_1:w_2+iw_4-w_6+iw_8:w_3-iw_5+w_7+iw_9]\;.$$

Therefore, we factored the rightmost vertical arrow of the diagram \eqref{diagramma1} as
$$\xymatrix{\C\otimes\H\times\sfera\ar[r]^-{\cong}&\C^4\times\CP^1\ar[r]^-{S_2}&\CP^9\ar[r]^-{\delta_1}&\CP^3\ar[r]^-{\rho_1}&\H}\;.$$

\begin{rem}It could surprise that the map $\delta_1$ does not depend on the choice of the basis $\{1,I,J,K\}$; however, for $q\in\H$, let us consider the $\C$-linear map $F_q:\C^2\to\C^2$ defined by
$$F_q(z,w)=\gamma^{-1}(\gamma(z,w)q)\;.$$
Then,
$$\delta_1([w_0:\ldots:w_9])=\begin{bmatrix}F_1 & 0 &0 &0 &0\\0& F_1 &F_I & F_J& F_K\end{bmatrix}\begin{bmatrix}[w_0:w_1]\\\vdots\\ [w_8:w_9]\end{bmatrix}\;.$$
Moreover, the fact that $I\gamma(z,w)=\gamma(iz,iw)$ and that $IJ=K=-JI$ uniquely identifies all the linear maps involved.
\end{rem}

If we expand the diagram \eqref{diagramma1} incorporating the two factorizations found above, we obtain the following diagram.
$$\xymatrix{\cU\times\sfera\ar[dddd]_\pi\ar@{^{(}->}[dr]\ar[rrrr]^-{\mathscr{F}=(F,\mathrm{Id}_{\sfera})} &&&&\C\otimes \H\times\sfera\ar[dddd]^\pi\ar[ld]^-{\cong}\\
& \C\times\CP^1\ar[d]_{S_1}\ar@{-->}[rrdd]^{\widetilde{F}}& & \C^4\times\CP^1\ar[d]^{S_2}& \\
& \CP^3\ar[d]_{\rho_1}& & \CP^9\ar[d]^{\delta_1}& \\
&\H& & \CP^3\ar[rd]^{\rho_1}& \\
V\ar[rrrr]^f\ar@{^{(}->}[ur]&  && & \H}$$

This proves Lemma \ref{lmm_twist}, stated in \S\ref{guiding_example} as a guiding example: $\widetilde{F}$ is defined by the commutative diagram above and can be expressed in terms of the function $\mathscr{F}$, or (which is equivalent) in terms of the function $F$ given by Ghiloni and Perotti; it is holomorphic, sends the complex spheres $\{z\}\times\CP^1$ into lines of $\CP^3$, and satisfies $f\circ\rho=\rho_1\circ\widetilde F$, this last identity being immediate from the commutativity of the diagram. This result is contained in \cite{GSS}.

By the Segre embedding $S_1$, we can identify $\cU\times\CP^1$ with an open set of the Klein quadric in $\CP^3$; moreover, we have the following corollary, also present in \cite{GSS}.

\begin{corol}\label{cor_twist}Given $V$, $\cU$ as before, every slice-regular function $f:V\to\H$ induces a holomorphic function
$$\mathcal{F}:\cU\to\mathrm{Gr}_\C(2, \C^4)\;.$$
\end{corol}

Both in Lemma \ref{lmm_twist} and Corollary \ref{cor_twist}, we could also impose some symmetry conditions on the holomorphic function, obtaining a bijective correspondence between slice-regular functions and this class of holomorphic functions. Those symmetry conditions can be written in the form of equivariance with respect to a real structure on $\CP^3$, or equivalently on the Grassmannian $\mathrm{Gr}_\C(2, \C^4)$. The detailed description can be found in \cite{GSS}.

\begin{rem}We want to stress that, whereas the functions $F$ and $\mathscr{F}$ are defined only in terms of $f$, the functions $\widetilde{F}$ and $\mathcal{F}$ depend on the choice of the orthonormal basis $\{1,I,J,K\}$.\end{rem}

From this detailed analysis, we resume the result from \cite{M1} about the complex-analyticity of $V(q)$: consider $V(0)$, which is the projection on the factor $\CP^4$ of the set $(\delta_1\circ S_2)^{-1}(L)$ where $L\subseteq\CP^3$ is the line $\{[w_0:w_1:0:0]\ :\ [w_0:w_1]\in\CP^1\}$. Indeed, from this description one recovers completely \cite{M1}*{Theorem 3.3} just working out explicitly the coordinate expressions of the maps involved.

\begin{rem}\label{rem_pvs}We describe the construction of $\tilde{F}$ from a different viewpoint.

We notice that, if $q=x+Iy$, then
$$Q_p(q)=Q_p(x+Iy)=p^{-1}(x+Iy)p=x+Q_p(I)y$$
so, given that $Q_p(\sfera)=\sfera$, we are looking at the set
$$\mathcal{V}=\{(v,z,w)\in \sfera\times\C\times \C\ :\ pv=Ip, \textrm{ with } p=\gamma(z,w)\}$$
together with its projection $p_1:\mathcal{V}\to\sfera$, which, as we noted before, is surjective. We showed that $p_1^{-1}(v)$ is a complex line in $\C^2$ and $p_1:\mathcal{V}\to\sfera$ is a line bundle on $\CP^1$.

The line bundle $p_1:\mathcal{V}\to\sfera$ admits a nowhere vanishing section $\sigma$, which we interpret as a function from $\sfera$ to $\C^2$ (given that $\mathcal{V}$ is a subbundle of the trivial bundle of rank $2$); given $F:\mathcal{U}\to\C^4$, we construct the function $\tilde{F}:\mathcal{U}\times\sfera\to \CP^3=\P(\C^2\oplus\C^2)$ by
$$\tilde{F}(z,v)=[(\gamma(\sigma(v)), \gamma(F_1(z)T_1(\sigma(v))+F_2(z)T_I(\sigma(v))+F_3(z)T_J(\sigma(v))+F_4(z)T_K(\sigma(v))))]\;,$$
where $T_v(z,w)=\gamma^{-1}(\gamma(z,w)v)$.
\end{rem}

\end{document}